\documentclass[11pt]{article}
\usepackage{amsmath,amsthm,amsfonts,amssymb,bm,wasysym}
\usepackage{epsfig}
\usepackage[usenames]{color}
\usepackage{verbatim}
\usepackage{hyperref}
\usepackage{multicol}
\usepackage{comment}
\usepackage{float}
\usepackage{graphicx}
\usepackage[utf8]{inputenc}
\usepackage[normalem]{ulem}

\parskip \medskipamount
\newtheorem{theorem}{Theorem}[section]

\numberwithin{equation}{section}
\newtheorem{lemma}[theorem]{Lemma}
\newtheorem{proposition}[theorem]{Proposition}
\newtheorem{corollary}[theorem]{Corollary}

\newtheorem{remark}[theorem]{Remark}

\numberwithin{equation}{section}

\def\Z{\mathbb{Z}}

\def\R{\mathbb{R}}

\def\E{\mathbb{E}}

\def\bP{\mathbb{P}}

\def\FF{\mathcal{F}}
\def\EE{\mathcal{E}}

\def\LL{\mathcal{L}}

\def\GG{\mathcal{G}}
\def\RR{\mathcal{R}}

\newcommand{\1}[1]{{\mathbf 1}{\{#1\}}}

\renewcommand{\phi}{\varphi}
\renewcommand{\epsilon}{\varepsilon}

\newcommand{\coz}[1]{\bar{#1}}
\newcommand{\cog}[1]{#1\text{\tiny $\updownarrow$}}

\begin{document}

\title{Once reinforced random walk on $\Z\times \Gamma$}

\author{Daniel Kious\thanks{ NYU-ECNU Institute of Mathematical Sciences at NYU Shanghai}\and Bruno Schapira\thanks{Aix-Marseille Universit\'e, CNRS, Centrale Marseille, I2M, UMR 7373, 13453 Marseille, France;  bruno.schapira@univ-amu.fr} \and Arvind Singh\thanks{University Paris Sud} 
}
\date{}
\maketitle
\begin{abstract}
We revisit Vervoort's unpublished paper \cite{Ver} on the once reinforced random walk, and prove that this process is recurrent on any graph of the form $\Z\times \Gamma$, with $\Gamma$ a finite graph, for sufficiently large reinforcement parameter.
We also obtain a shape theorem for the set of visited sites, and show that the fluctuations around this shape are of polynomial order. The proof involves sharp general estimates on the 
time spent on subgraphs of the ambiant graph which might be of independent interest.   
MSC 2010 \emph{subject classifications:} 60K35.
\end{abstract}

\section{Introduction}\label{sec:intro0}
\subsection{General overview}
The once-reinforced random walk (ORRW) belongs to the large class of self-interacting random walks, whose future evolution depends on its past history. The study of these processes is usually difficult and basic properties such as recurrence and transience are hard to obtain. One of the most famous example of self-interacting random walks is the linearly edge-reinforced random walk introduced by Coppersmith and Diaconis \cite{CD} in the eighties, for which recurrence and transience were only recently proved in a series of papers, see \cite{ACK,ST,SZ,DST}.

Even though the definition of ORRW, introduced in 1990 by Davis \cite{Dav90}, is simple, it turns out that its study does not seem easier than that of the linearly reinforced RW, and results on graphs with loops are very rare.
In this model, the  current weight of an edge is $1$ if it has never been crossed and $1+\delta$ otherwise, with $\delta>0$.  It has been conjectured by Sidoravicius that ORRW is recurrent on 
$\mathbb{Z}^d$ for  $d\in\{1,2\}$ and undergoes a phase transition for $d\ge3$, being recurrent when the parameter $\delta$ is large and transient when it is small. These questions on 
$\mathbb{Z}^d$, $d\ge2$, are completely open. Until recently, it was not even clear whether this weak reinforcement procedure could indeed change the nature of the walk, so that ORRW could 
be recurrent on a graph which is transient for simple random walk, as soon as the parameter $\delta$ is large enough. The first example of such phase transition was provided in \cite{KS16} 
on a particular class of trees with polynomial growth, which is in contrast with the result of Durrett, Kesten and Limic \cite{DKL02} who showed that the ORRW is transient on regular trees for any 
$\delta>0$ (later generalized to any supercritical tree by Collevecchio \cite{Coll}). More recently, the complete picture on trees has been given in \cite{CKS}: the critical parameter of ORRW on a 
locally finite tree is equal to its {\it branching-ruin number}, which is defined in \cite{CKS} as a polynomial regime of the branching number (see \cite{L90}).

As already mentioned, results on graphs with loops are very rare. Sellke \cite{Sellke} first investigated the case of the ladder $\mathbb{Z}\times\{1,...,L\}$, with $L\ge 3$, and 
proved that the ORRW is almost surely recurrent on this graph for all $\delta\in(0,1/(L-2))$. The proof is a simple consequence of a general (and nice) martingale argument, 
but it does not really face the difficulty of possible drift pushing the walk 
toward infinity, which can in principle happen in the presence of loops in the graph. In an unpublished paper, Vervoort \cite{Ver} announced a more difficult result, namely that the ORRW is recurrent on the ladder for all large enough reinforcement parameter $\delta$. 
Unfortunately, his proof was never published and the preprint \cite{Ver} is unpolished, with gaps and mistakes. 
The general strategy of \cite{Ver} was to show that the mean drift of the walk, each time it exits its present range, is almost balanced. 
The reason being that for large enough $\delta$, all possible exit edges are equally likely to be chosen. Thus, at least at first order, there should be equal probability to get a drift   
$\delta$ to the right (when the exit edge is an horizontal edge oriented to the left) as an opposite drift (when this edge is oriented to the right). 
However, an important ingredient, which was missing in \cite{Ver}, is to show first that the ORRW cannot travel a large distance before exiting its present range. 
Indeed the first order 
approximation of uniformity for the choice of the exit edge is only valid when the edges taken into consideration are not too far one from each other. 
One difficulty then is to obtain an estimate, which is uniform over all the possible ranges (or finite subgraphs of $\mathbb Z\times \Gamma$). 
We prove such general result here, which might be of independent interest, with the help of electrical network techniques. 
Details can be found in our Proposition \ref{outbound} below.

Furthermore, we improve the lower bound on $\delta$, and obtain a polynomial bound in the height of the ladder, instead of an exponential one, which was implicit in \cite{Ver}. 
For this purpose, one needs 
to adapt the notion of walls from \cite{Ver}, to ensure their existence with a probability $1/2$, instead of an exponentially small (in the size of $\Gamma$) one. 
We also analyze the fluctuations of the range of the walk and provide a shape theorem. Finally we show that the successive return times to the origin have finite expectation.

\subsection{Model and results}\label{sec:intro}
Let us define a nearest-neighbor random walk $(X_n)_{n\ge 0}$ as a  ORRW  on a (nonempty, locally finite and undirected) graph $G$, with reinforcement parameter $\delta\ge 0$.
First, the {\it current weight} of an edge is defined as follows: at time $n$, an edge has conductance $1$ if it has never been crossed (regardless of any orientation of the edges) and conductance $1+\delta$ otherwise. For any $n\ge0$, let $E_n$ be the set of non-oriented edges crossed up to time $n$, that is
\begin{eqnarray}
\label{defEn}
E_n:=\left\{\{x,y\}: \ x,y\in G\text{ and there exists }  1\le k\le n, \text{ such that~}\{X_{k-1},X_k\}=\{x,y\}\right\}.
\end{eqnarray}
At time $n\in\mathbb{N}$, if $X_n=x\in G$, then the walk jumps to a neighbor $y$ of $x$ with conditional probability
\begin{equation}\label{def.ORRW}
\mathbb{P} \left[X_{n+1}=y\mid \mathcal{F}_n\right]=\frac{\delta\1{\{x,y\}\in E_n}+1}{\sum_{z:z\sim x}\left(\delta\1{\{x,z\}\in E_n}+1\right)},
\end{equation}
where $\left(\mathcal{F}_n\right)_{n\ge 0}$ is the natural filtration generated by the walk, i.e.~$\mathcal{F}_n=\sigma(X_0,\dots,X_n)$.

Our first result is the following: 
\begin{theorem}\label{theo1}
There exists a constant $C>0$, such that for any finite connected graph $\Gamma$, 
the once-reinforced random walk on $\Z\times\Gamma$ is recurrent for any reinforcement parameter $\delta\ge C| \Gamma|^{40}$. 
\end{theorem}
Note that here by {\it recurrent} we mean that almost surely every site is visited infinitely often.

Our second result is a shape theorem. 
Denote by $\RR_n$ the graph whose vertex set is $\{X_0,\dots,X_n\}$, the set of visited sites up to time $n$, and whose edges are those crossed by the walk up to this time. Let $t(n)$ be the first time when the number of vertices in this graph equals $(2|\Gamma|+1)n$.

\begin{theorem}\label{theo2}
There exists a constant $C>0$, such that for any finite connected graph $\Gamma$ and any $\delta\ge C |\Gamma|^{40}$, the following holds:  
almost surely for all $n$ large enough, there exists $x_n \in \Z$, such that
$$\{x_n-n+n^{1/\delta^{1/8}},\dots,x_n+n-n^{1/\delta^{1/8}}\}\times \Gamma \ \subseteq \ \RR_{t(n)}\ \subseteq \ \{x_n-n-n^{1/\delta^{1/8}},\dots,x_n+n+n^{1/\delta^{1/8}}\}\times \Gamma,$$
where inclusions here are meant as inclusion of graphs.
\end{theorem} 

\begin{remark}\emph{
We do not expect that the center of the cluster $x_n$ could be taken to be zero. Indeed for the ORRW on $\Z$ (i.e. when $\Gamma$ is reduced to a single vertex), explicit computations show that, for any $\epsilon>0$, $x_n/n<-1/2+\epsilon$ infinitely often and  $x_n/n>1/2-\epsilon$ infinitely often.
} 
\end{remark}
\begin{remark}\emph{Concerning the exponent $\delta^{-1/8}$, it is far from being optimal. Our proof would allow to replace the constant $1/8$ by any other constant smaller than  $1/4$, at the cost of imposing larger 
$\delta$. But we do not believe that this would be optimal neither. In fact we expect that the correct order of the fluctuations is precisely in $n^{\frac{1+o(1)}{\rho}}$, with $\rho$ the asymptotical mean drift per level: $\rho:=\lim_{x\to +\infty} \E(D_x)$, 
where $D_x$ equals $\delta$ times the number of edges between level $x$ and $x+1$ which are crossed for the first time from left to right minus the number of those edges crossed for the first time in the other direction. But here the main issue would be to show that the above limit actually exists. On the other hand, our proofs in this paper show that if the limit indeed exists, then it is larger than $\delta$, up to lower order terms (and in particular it is positive for large $\delta$). We also suspect that $\rho$ should be equivalent to $c\delta$, for some constant $c\ge 1$, when $\delta$ goes to infinity. 
} 
\end{remark}

Note that contrarily to Theorem \ref{theo1}, Theorem \ref{theo2} does not hold for the simple random walk (which corresponds to the case when $\delta=0$). 
In this direction, we also show at the end of the paper that the successive return times to the origin have finite expectation. 

\subsection{Organization of the paper}
The paper is organized as follows. In Section \ref{sec:network}, we prove general estimates on random walks on networks where conductances take only two values: $\delta$ on a finite subgraph $A$ of $\Z \times \Gamma$ and one elsewhere. Our main results there are estimates, which are uniform on $A$, on the time spent on certain level sets, that is subsets of the form $\{i\}\times \Gamma$. This section can be read independently of the rest of the paper, and might be interesting on its own. 
Then in Section \ref{section:Dwall}, we define a notion of wall, that extends the one from Vervoort's paper \cite{Ver}. The interest of this new definition is to obtain bounds of polynomial order (in the size of $|\Gamma|$) in all our results. In section \ref{sec:gambler} we gather the results proved so far to obtain gambler's ruin type estimates. These estimates are then used in Section \ref{sec:prooftheo} to prove Theorems \ref{theo1}, and \ref{theo2}, and in Section \ref{sec:returntimes} to prove that the successive return times to the origin have finite expectation.  

\section{Random walks on sub-graphs of $G = \Z \times \Gamma$. }\label{sec:network}

This section gathers some results concerning (reversible) random walks on sub-graphs of $\Z \times \Gamma$ where $\Gamma$ is a finite graph. In particular, we study the position where such a walk exits a given sub-graph. As such, the section does not deal specifically with once-reinforced random walk but the results obtained here will play a crucial role during the proof of the main theorems. We also believe that some estimates such as Proposition \ref{propDiffusif} may be found of independent interest. 

\subsection{Notation}
A graph $G=(V,E)$ is a collection of vertices $V$ and edges $E$. By a small abuse of notation, we shall sometimes identify a graph and its set of vertices when the associated edge set is unambiguous. An undirected edge between two vertices $x$ and $y$ is denoted by $\{x,y\}$, while a directed edge is denoted by $(x,y)$. We write $x\sim y$ when $\{x,y\}\in E$ and we say in this case that $x$ and $y$ are neighbors. All the graphs considered here are assumed to be non-empty and locally finite, meaning that all vertices have only finitely many neighbors. If $e=(x,y)$ is a directed edge, we call $x$ the tail and also denote it by $e^-$, and $y$ the head and denote it by $e^+$. We write $\vec E$ for the set of directed edges of $G$. 

Given two vertices $x$ and $y$, we denote by $d(x,y)$ their graph distance in $G$. For a subgraph $A\subseteq G$, we denote by $d_A(x,y)$ the induced (also called intrinsic) distance, i.e. the minimal number of edges needed to be crossed to go from $x$ to $y$ while staying inside $A$. In particular, we have $d_G$  = $d$. 

Given a subgraph $A = (V_A, E_A) \subseteq G$, we define 
\begin{align}
\label{defpartialsE}&\partial_e A :=\{e\in \vec E:\ e^-\in V_A,\, \{e^-,e^+\}\notin E_A\}\\
\label{defpartialsV}&\partial_vA :=\{v\in V_A:\ \exists e\in \partial_e A\text{ with~}e^-=v\}.
\end{align}
In words, $\partial_e A$ is
the set of directed edges of $G$ which do not belong to $A$ as an undirected edge but whose tail belongs to $A$. Notice that the head of a directed edge $e\in\partial_e A$ may, or may not, be in $A$.
The set $\partial_vA$ is the set of  tails of the edges in $\partial_eA$, or equivalently the set of vertices adjacent to an edge outside $A$.

In this paper, we  consider cylinder graphs of the form $G = \Z \times \Gamma$. In this case, if $a\in G$, we will denote by $\coz{a}$ and $\cog{a}$ the respective projections on $\Z$ and $\Gamma$ so that 
$$a := (\coz{a},\cog{a}).$$ 
Finally, for $z\in \Z$, we denote by $\{z\}\times \Gamma$ the sub-graph of $G$ isomorphic to $\Gamma$ with edge set consisting of all edges with both endpoints in $\{z\}\times \Gamma$. We call this sub-graph the {\bf level set $z$}.

\subsection{Reversible RW and electrical networks}

We recall here some standard results on random walks and electrical networks which we will use repeatedly in the paper. We refer the reader to \cite{DoyleSnell}  and \cite{LP} for a comprehensive and thorough presentation of the theory.
  
A \textbf{network} is a graph $G=(V,E)$, endowed with a map $c:E\to (0,\infty)$. The value $c(e)$ of an edge $e$ is called its \textbf{weight} or \textbf{conductance}, and its reciprocal $r(e):=1/c(e)$ is called its \textbf{resistance}. A random walk  on a network $(V,E,c)$ is the Markov chain which moves only to neighbors of its current position, choosing it with a probability proportional to the weight of the edge traversed. We denote the law of the chain starting from $a\in V$ by $\bP_a$. This process is reversible with respect to the measure $\pi$ defined by $\pi(x) := \sum_{y\sim x} c(\{x,y\})$. Given a subset of vertices $V_0\subseteq V$, a map $h:V\to \R$ is said to be harmonic on $V_0$ if it satisfies:
$$
h(x)=\frac{1}{\pi(x)}\sum_{y\sim x} c(\{x,y\})h(y)\quad\hbox{for any $x\in V_0$.}
$$
Given a vertex $a\in V$ and a subset $Z\subseteq V\setminus\{a\}$, 
a {\bf voltage} $\mathbf{v}$ is a function which is harmonic outside $\{a\}\cup Z$, and vanishes on $Z$. 
Given a voltage function, we define the associated {\bf current} function $\mathbf{i}$ on the oriented edges by 
\begin{flalign*}
&&&& \mathbf{i}(x,y) := c( \{x,y\})[\mathbf{v}(x)-\mathbf{v}(y)]. &&  \hbox{(\textbf{Ohm's law})}
\end{flalign*}
Then, $\mathbf{i}$ is a \textbf{flow} from $\{a\}$ to $Z$, which means an anti-symmetric function on the set of oriented edges $\vec E$ such that
\begin{flalign*}
&&&&&&&&& \sum_{y\sim x} \mathbf{i}(x,y) = 0\quad\hbox{for all $x\in V \setminus \{a\}\cup Z$.} &&  \hbox{(\textbf{Kirchoff's node's law})}
\end{flalign*}
The \textbf{strength} of the flow is defined by $\|\mathbf{i}\| := \sum_{y\sim a} \mathbf{i}(a,y) = - \sum_{z\in Z}\sum_{y\sim z} \mathbf{i}(z,y)$. We say that we have a {\bf unit current} flowing from $a$ to $Z$ when $\|\mathbf{i}\| = 1$, and one defines similarly a {\bf unit flow}. 

Given a random walk $(S_n)_{n\ge 0}$, the hitting time of a set of vertices $B\subseteq V$ is defined by
\begin{equation*}
H_B := \min\{n\ge 0\, :\, S_n\in B\},
\end{equation*}
whereas the first return time is defined by
\begin{equation*}
H_B^+ :=\min\{n\ge 1\,:\, S_n\in B\}.
\end{equation*}
To simplify notations, we just write $H_b$ (resp $H_b^+$) when $B = \{b\}$. Similarly, when $G$ is of the form $\Z\times \Gamma$, we also use the notation $H_r := H_{ \{r\} \times \Gamma}$ for the hitting time of the level set $r\in\Z$. 

The {\bf effective conductance} $\mathcal C(a \leftrightarrow Z)$ between a vertex $a$ and a subset $Z\subseteq V$ is defined by the formula
\begin{equation}\label{conductance}
\mathcal C(a \leftrightarrow Z) := \pi(a) \bP_a[H_Z<H_a^+].
\end{equation}
Its reciprocal is called the {\bf effective resistance} between $a$ and $Z$ and is denoted by $\mathcal R(a  \leftrightarrow Z)$.  
It follows from the maximum principle that there exists a unique unit current flowing from $a$ to $Z$. The corresponding voltage is the unique function that is harmonic outside of $\{a\}\cup Z$, vanishes on $Z$, and satisfies $\mathbf{v}(a)=\mathcal R(a\leftrightarrow Z)$. 

We recall three important results which we will need in later sections.
\begin{proposition}[{\bf Current as edge crossings}, Prop. 2.2 in \cite{LP}]\label{unitcurrent}
Let $G$ be a finite connected network. Consider the random walk starting at some vertex $a$ and let $Z$ be a subset of vertices not containing $a$. For $x\sim y$, let $N_{xy}$ be the number of crossings of the directed edge $(x,y)$ by the walk before it hits $Z$. We have $\E_a[N_{xy}-N_{yx}] = \mathbf{i}(x,y)$, where $\mathbf{i}$ is the unit current flowing from $a$ to $Z$. 
\end{proposition}
As a consequence of this proposition (\emph{c.f.} Exercice 2.37 of \cite{LP}), if $\mathbf{i}$ is a unit current from $a$ to $Z$, then necessarily
\begin{equation}\label{courrant}
|\mathbf{i}(x,y)| \ \le \ 1\qquad \text{for all $x \sim y$.}
\end{equation} 
Given a flow $\mathbf{j}$ on an electrical network, the energy dissipated by the flow is defined by
$$
\mathcal{E}(\mathbf{j}) := \frac 12\sum_{e\in\vec{E}} r(e)\mathbf{j}^2(e)
$$
The following result characterizes the current among all flows on a network. 
\begin{proposition}[{\bf Thomson's principle}, p.~35 in \cite{LP}]\label{thomson}
The unit current $\mathbf{i}$ has minimal energy among all unit flows:
$$
\mathcal{E}(\mathbf{j}) > \mathcal{E}(\mathbf{i}) = \mathcal R(a  \leftrightarrow Z)\quad\hbox{for any unit flow $\mathbf{j} \neq \mathbf{i}$.}
$$
\end{proposition}
We say that a flow $\mathbf{j}$ has a \textbf{cycle} if there exist oriented edges $e_1,\ldots e_n$ with $e_i^+ = e^-_{i+1}$ and $e_n^+ = e_1^-$ and $\mathbf{j}(e_i) >0$ for all $i \in \{1,\ldots,n\}$. It follows from Thomson's principle that a \textbf{current $\mathbf{i}$ cannot have a cycle} because we could otherwise decrease its energy by removing from it a small flow with support on the cycle. Another immediate consequence of Thomson's principle is that the effective conductance/resistance is a monotone function of the conductances on the edges. 
\begin{proposition}[{\bf Rayleigh's Monotonicity Principle}, p.~36 in \cite{LP}]
\label{Rayleigh}
Let $G$ be a finite connected graph with two conductances assignments, $c$ and $c'$ such that $c\le c'$. 
Let $a$ be a vertex and $Z$ a subset of vertices not containing $a$. We have $\mathcal C_c(a\leftrightarrow Z)\le  \mathcal C_{c'}(a\leftrightarrow Z)$. 
\end{proposition}

We end this section with the remarkable {\bf Commute-Time Identity}, which relates the hitting times between two points of a graph and the effective resistance between these two points.

\begin{proposition}[{\bf Commute-Time Identity}, Corollary 2.21 in \cite{LP}]
\label{commute}
Let $G$ be a finite connected network and let $a$ and $z$ be two vertices of $G$. The commute time between $a$ and $z$ is
\[
\mathbb{E}_a\left[H_z\right]+\mathbb{E}_z\left[H_a\right]= \mathcal{R}(a\leftrightarrow z)\times \left(\sum_{x\in V}\pi(x)\right).
\] 
\end{proposition}

\subsection{The exit edge for a random walk on a sub-graph of $G = \Z \times \Gamma$}

By definition, the once-reinforced random walk behaves as a usual random walk as long as it stays inside its trace. More precisely, assume that at some time $n$, the ORRW has crossed exactly all the edges of a sub-graph $A\subseteq G$ (in particular $S_n\in A$). Then, from time $n$ and until it exits the sub-graph $A$, the ORRW behaves as the random walk on the electrical network $G$ with conductances given by 
\begin{equation}\label{ORRWnetwork}
c(e) := 1+\delta\mathbf{1}_{\{e\in A\}}.
\end{equation}
In particular, when $\delta$ is large, the probability to choose a non-reinforced edge is small. Thus, informally, one can visualize the walk as ``bumping'' on the boundary of its trace many times before exiting, and so it should ``mix'' a little more than the usual random walk. This remark leads to a key idea which originates from Vervoort \cite{Ver}: when $\delta$ is large, the distribution of the exit edge gets close (locally, in some sense) to the uniform measure on the boundary $\partial_e A$. 

In this subsection, we give two results in this direction that concern the distribution of the exit edge. They are stated in term of the random walk on the electrical network \eqref{ORRWnetwork} but they translate readily to the ORRW as explained above. The first result states that two edges on the boundary which are not too far away  have approximately the same probability to be chosen as the next exit edge.
\begin{proposition}\label{prop:balance}
Let $A$ be a finite connected sub-graph of a graph $G$. Fix $\delta >0$ and consider the electrical network on $G$ with conductances $c(e):=1+\delta\mathbf{1}_{\{e\in A\}}$. Let $(S_n)_{n\ge 0}$ denotes a random walk on this electrical network and define $\sigma$ as the first time the walk exits the sub-graph $A$:
\begin{eqnarray}\label{def:sigma}
\sigma:=\inf\{n\ge 1\, :\, (S_{n-1},S_n)\in \partial_eA\}.
\end{eqnarray}
For any $f_1,f_2\in \partial_eA$, and for any $a\in A$, we have 
\begin{equation}\label{bound:balance}
\left| \bP_a[ (S_{\sigma-1},S_\sigma)= f_1] - \bP_a[(S_{\sigma-1},S_\sigma) = f_2 ]\right| \ \le \  \frac{d_A(f^-_1,f^-_2)}{1+\delta}.
\end{equation}
\end{proposition}

\begin{proof}
Consider the finite connected graph $A^\Delta$  whose vertex set is $A\cup \{\Delta\}$, with $\Delta$ an additional cemetery vertex, and whose edge 
set consists of all the original edges in $A$ plus, for each $e\in \partial_e A$, one additional  edge between $e^-$ and $\Delta$. The edges inside $A$ are assigned weight $1+\delta$ whereas the edges adjacent to $\Delta$ are assigned unit weight. Note that this construction may create multiple edges between some vertex in $A$ and $\Delta$.    

Let $f_1,f_2\in\partial_e A$ with tails  $x_1=f_1^-$ and $x_2=f_2^-$. By construction, the law of $S$ up to time $\sigma$, matches the law of the random walk on the network $A^\Delta$, up to the hitting time of $\Delta$. Thus, according to Proposition \ref{unitcurrent} and using Ohm's law, for any $a\in A$, we have
\[
 \bP_a[(S_{\sigma-1},S_\sigma) = f_i ] = \mathbf{i}(f_i) = c(f_i)\big(\mathbf{v}(x_i) - \mathbf{v}(\Delta)\big) = \textbf{v}(x_i), 
\]
where $\textbf{v}$ is the voltage at $x_i$ when a unit current $\mathbf{i}$ flows from $a$ to $\Delta$. 
By definition there exists a path of length $d_A(x_1,x_2)$ inside $A$ going from $x_1$ to $x_2$ and composed of edges with conductance $1+\delta$. Applying Ohm's law along this path and using \eqref{courrant}, we find that
\[
|\mathbf{v}(x_1) - \mathbf{v}(x_2)| \leq \frac{d_A(x_1,x_2)}{1+\delta},
\]
and the result follows. 
\end{proof}

The proposition above is fairly general since it does not make any assumption on the graph $G$ (it need not be of cylinder type). However, as time increases, so does the size of the boundary of the trace of the walk. Thus, without additional information on the distribution of the exit probabilities, the bound \eqref{bound:balance} applied to the ORRW becomes mostly useless when the number of possible exit edges becomes much larger than $\delta$ . 

In order to keep \eqref{bound:balance} relevant, we need to control the number of exit edges which have a non negligible probability of being chosen and show that they are $o(\delta)$.  This estimate which is missing from Vervoort's paper is the purpose of the next proposition. Unlike Proposition \ref{prop:balance}, it is specific to cylinder graphs.
\begin{proposition}\label{outbound}
Let $A$ be a finite connected sub-graph of $G = \Z \times \Gamma$ where $\Gamma$ is a finite connected graph. Fix $\delta>0$ and consider the random walk $S$ on the electrical network $G$ with conductances $c(e) := 1 + \delta\mathbf{1}_{\{e \in A\}}$. Fix $a\in A$ and suppose that there exist $d$ integers $\{ s_1, s_2, \ldots, s_d\}$ and $r$ such that 
\begin{enumerate}
\item $\coz{a} < s_1 <\dots < s_d < r$.
\item For each $i \in \{1,\ldots, d\}$, there exist $x \in \partial_{v}A$ with $\coz{x} = s_i$ (there is an exit edge at each level).
\end{enumerate}
Recall that $\sigma$ defined by \eqref{def:sigma} denotes the first time the walk exits the sub-graph $A$ and that $H_{r}$ denotes the first time it reaches level $r$. We have
\begin{equation}\label{eq:diffoutedge}
\bP_a\big( H_{r} < \sigma \big) \leq 5\exp\left(-\frac{1}{4^4|\Gamma|^3}\left(\frac{d^2}{1+\delta}\right)^{\frac{1}{3}}\right)
\end{equation}
\end{proposition}
The proposition above tells us that the random walk on the network \eqref{ORRWnetwork} cannot travel too far away horizontally without exiting its trace. More precisely, the number of incomplete levels it can cross before exiting is (at most) of order $\sqrt{\delta}$. This means that, for the ORRW, only the exit edges belonging to the nearest $\sqrt{\delta}$ incomplete levels have to be taken into account. But then, there are no more than $|\Gamma|^2\sqrt{\delta} \ll \delta$ such exit edges (because the underlying graph $G$ is a cylinder). Thus, we are now in the case where we can use Proposition \ref{prop:balance} to control the exit probabilities.

Let us also remark that the ratio $d^2/(1+\delta)$ in \eqref{eq:diffoutedge} is not surprising because the walk on any sub-graph $A$ of $G = \Z \times \Gamma$ is diffusive. Thus, we can expect that it spends a time of order $d^2$ inside a slice of diameter $d$. On the other hand, each time the walk visits a site on the boundary of $A$, it has a probability proportional to $1/(1+\delta)$  to exit it at the next step. This heuristic is simple but making it rigorous is challenging because the upper bound \eqref{eq:diffoutedge} needs to hold uniformly on all possible sub-graph $A$. The proof we present here is rather convoluted and will by carried out at the end of the next subsection.

\subsection{Proof of Proposition \ref{outbound}}

In this section, we need to consider random walks on different graphs. In order to distinguish between these processes, we will use super-script that refer to the underlying graph. For instance, given a sub-graph $A$ of $G$, probabilities relating to a random walk on $A$ will be denoted by  $\bP^A(\cdot)$ whereas we will keep the usual notation $\bP(\cdot)$ for  a random walk on the whole graph $G$. 

In everything that follows, $\Gamma$ denotes a finite connected graph. We start with a simple lemma which bounds the return time to a given vertex on the same level as the starting position for the simple random walk on $\Z\times \Gamma$.
  
\begin{lemma} \label{lem:hitfrontAx_new}
Consider the simple random walk on $G = \Z\times \Gamma$. Let $a,b \in G$ such that $\coz{a} =  \coz{b}$. Recall that $H_b$ 
denotes the hitting time of vertex $b$. We have  
$$\bP_a\left(H_{b} \, \le\,  4^6|\Gamma|^6\right)\  \ge\ \frac{1}{2}.$$
\end{lemma}
Recall that, for a subgraph $A\subset G $,  $\partial_vA$, defined in \eqref{defpartialsV}, denotes the set of vertices of $A$ which are adjacent to an edge outside $A$. Since we can couple the simple random walk on $G$ and the simple random walk restricted on $A$ so that they coincide until they reach a vertex of $\partial_vA$, the previous lemma directly entails
\begin{corollary}\label{cor:hitfrontA_new}
Let $A$ be a sub-graph of $G = \Z\times \Gamma$. Consider the simple random walk on $A$, i.e.~on the network $G$ with conductances $c(e)=\mathbf{1}_{\{e\in A\}}$. For any $a\in A$ such that $\partial_vA$ contains at least one vertex at level $\coz{a}$, we have
$$\bP^A_a\left(H_{\partial_vA} \, \le\,  4^6|\Gamma|^6\right)\  \ge\ \frac{1}{2}.$$
\end{corollary}
\begin{proof}[Proof of Lemma \ref{lem:hitfrontAx_new}]
Without loss of generality, we can assume that $\coz{a} = \coz{b} = 0$. Fix $L = 4^4 |\Gamma|^3$ and consider the graph $G_L := [|-L,L|]\times\Gamma$ on which we put unit conductances (\emph{i.e.} we are considering the SRW on the graph). Noticing that the number of oriented edges in $G_L$ is bounded by $(2L + 1)|\Gamma|(|\Gamma|-1) + 4L|\Gamma| \leq 4L|\Gamma|^2$, Proposition \ref{commute} (the commute-time identity) shows that 
$$\E^{G_L}_{a}(H_b) \le 4L|\Gamma|^2\RR_{G_L}(a\leftrightarrow b).$$ Furthermore, the effective resistance between $a$ and $b$ is bounded by the graph distance between those two vertices. This may be checked, for instance, from Rayleigh's monotonicity principle (Proposition \ref{Rayleigh}) by putting null conductances everywhere except on a geodesic path between the two vertices. Since $a$ and $b$ are on the same level, we deduce that $\E^{G_L}_{a}(H_b) \le 4L|\Gamma|^3$. Set $T := 4^2 L|\Gamma|^3 = 4^6|\Gamma|^6$. Using  Markov's inequality, we find that 
\[
\bP^{G_L}_a\left(H_{b}\le T\right)\ge\frac{3}{4}.
\]
On the other hand, if $S$ denotes a simple random on $\Z$ starting from $0$, an application of the reflection principle shows that 
$$
\bP\Big(\max_{k\leq T} |S_k| \geq L\Big)  = \bP\Big(\max_{k\leq T} |S_k| \geq 4\sqrt{T}\Big) \leq 4\bP\Big(S_T \geq 4\sqrt{T}\Big) \leq  4\frac{\E[S^2_T]}{(4\sqrt{T})^2} = \frac{1}{4}
$$
Thus, the  previous inequality shows that the probability that the simple random walk on $G_L$ starting from $a$ hits level $L$ or $-L$ before time $T$ is at most $1/4$. But, until this happens, the random walks on $G$ and on $G_L$ coincide so we conclude that
$$
\bP_a\left(H_b \leq T\right) \;\geq\; \bP^{G_L}_a\left(H_b \leq T < H_{-L} \wedge H_{L}\right) \;\geq\; \bP^{G_L}_a\left(H_b \leq T\right) - \frac{1}{4} \;\geq\; \frac{1}{2}.
$$
\end{proof}  

\begin{proposition}\label{propDiffusif}
Let $A$ be a finite connected sub-graph of $G=\Z\times \Gamma$. Fix $a \in A$, and $d\geq 1$. Let $\mathcal{S} = \{ s_1, s_2, \ldots, s_d \}$ and $r$ be integers such that
$$
\coz{a} < s_1 < s_2 < \ldots < s_d < r, \quad\hbox{and}\quad A \cap (\{r\} \times \Gamma) \neq \emptyset.
$$
Consider the simple random walk $S$ on $A$ started from $a$. Recall that $H_{r}$ denotes the hitting time of the level set at $r$. This stopping time is a.s. finite since $A$ is finite and has a vertex at level $r$. Let $\mathcal{L}^r_{\mathcal{S}}$ be the total time spent on the levels of $\mathcal{S}$ before reaching level $r$:
$$
\mathcal{L}^r_{\mathcal{S}} :=  |\{n \leq H_{r}\;:\; \coz{S_n} \in \mathcal{S}\}|
$$
We have
$$
\bP^A_a\big( \mathcal{L}^r_{\mathcal{S}} < k\big) \leq \frac{|\Gamma| \sqrt{ 8 k}}{d}\quad\hbox{for any $k\geq 1$.}
$$
\end{proposition}

\begin{figure}
\begin{center}
\includegraphics[height=3.7cm]{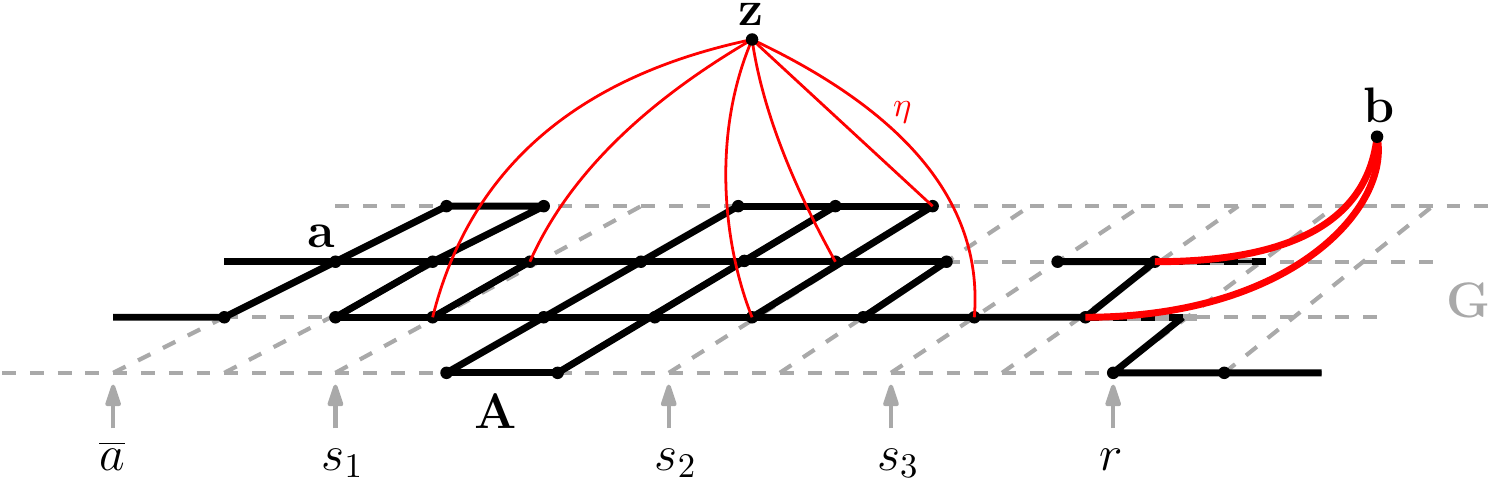}
\end{center}
\caption{\label{figTildeA}\small{Example of a modified graph $\tilde{A}$ on $G = \Z\times \{1,2,3,4\}$ with $d=3$. The modification are marked in red. The fat edges have conductance $1$ whereas the thin edges adjacent to $z$ have conductance $\eta$.}}
\end{figure}

\begin{proof} 
We construct a modified electrical network $\tilde{A}$ in the following way. First, we put unit conductances on each edge of $A$. Then, we glue all vertices of $A$ at the final height $r$ together and call the resulting vertex $b$. We also fix $\eta >0$ and create a new vertex $z$ connected by an edge of conductance $\eta > 0$ to every vertices $x\in A$ such that $\coz{x} \in \mathcal{S}$ (\emph{c.f.} figure \ref{figTildeA} for an illustration). Consider now a unit current $\mathbf{i}$ flowing from $a$ to $\{b,z\}$. Let 
$$\mathbf{i}(b) := \sum_{x\sim b} \mathbf{i}(x,b),$$ 
be the total current flowing into the sink vertex $b$. For $k=1,\ldots,d$, let also 
$$
\mathbf{i}_k := \sum_{x\in A\; : \; \coz{x} = s_k} \mathbf{i}(x,z),
$$
denotes the total current flowing from level $s_k$ to $z$. From a probabilistic point of view, $\mathbf{i}(b)$ is the probability that the walk on the electrical network $\tilde{A}$ started from $a$ hits $b$ before hitting $z$. Similarly, $\mathbf{i}_k$  is the probability that the walk hits $z$ before $b$ while exiting through one of the edges added at level $s_k$. We now show that
\begin{equation}\label{ik_j_bound}
\mathbf{i}_k \geq \frac{(d-k + 1) \eta \mathbf{i}(b) }{|\Gamma|^2} \quad\hbox{for any $k = 1,\ldots, d$.}
\end{equation}
To do so, we will need the following lemma which provides a decomposition of a flow without cycle on arbitrary graphs.
\begin{lemma}\label{lemmaFlow} Consider a finite connected graph A and fix three vertices $a,b,z\in A$. Let $\mathbf{i}$ be a flow from $a$ to $\{b,z\}$ such that
\begin{enumerate}
\item For all $x\in A$, we have $\mathbf{i}(a,x) \geq 0$ (source) and $\mathbf{i}(z,x) \leq 0$ and $\mathbf{i}(b,x) \leq 0$ (sinks). 
\item the flow $\mathbf{i}$ does not have any cycle. 
\end{enumerate}
Then, there exist a flow $\mathbf{j}$ on $A$ from $a$ to $b$ such that 
\begin{enumerate}
\item[(a)] $\mathbf{j}(x,b) = \mathbf{i}(x,b)$ for any $x \in A$. Therefore, $\mathbf{j}$ is a flow of total strength $\|\mathbf{j}\| = \mathbf{j}(b) = \mathbf{i}(b) \leq \|\mathbf{i}\|$.
\item[(b)] $\mathbf{j}(x,z) = 0$ for all $x\in A$ (nothing flows in $z$). 
\item[(c)] For any $x,y \in A$, $\mathbf{i}(x,y) \mathbf{j}(x,y) \geq 0$ (the flows $\mathbf{i}$ and $\mathbf{j}$ have the same direction). 
\item[(d)] For any $x,y \in A$, $|\mathbf{j}(x,y)| \leq |\mathbf{i}(x,y)|$.
\item[(e)] For any $x,y \in A$, $|\mathbf{j}(x,y)| \leq \|\mathbf{j}\| = \mathbf{i}(b)$.  
\end{enumerate}  
(here and below, we use the convention that, for any flow, $\mathbf{j}(x,y) = 0$ if $x$ and $y$ are not neighbors).
\end{lemma}
We postpone the proof of the lemma to finish that of the proposition. Fix $k\in \{1,\ldots,d\}$ and $0 < \varepsilon < 1$. Since the current $\mathbf{i}$ fulfills assumptions 1. and 2. of the lemma, we can consider the flow $\mathbf{j}$ as above and use it to create a new unit flow $\mathbf{i}^{k,\varepsilon}$ from $a$ to $\{b,z\}$ where some of the original current flowing into $b$ is diverted toward $z$ by going through the edges of conductance $\eta$ added at level $s_k$. More precisely, we set, for $x,y \in A - \{z\}$, 
$$
\mathbf{i}^{k,\varepsilon}(x,y)  := 
\begin{cases}
\mathbf{i}(x,y)  & \hbox{if both $\coz{x} \leq s_k$ and $\coz{y} \leq s_k$}\\
\mathbf{i}(x,y) - \varepsilon \mathbf{j}(x,y) & \hbox{if either $\coz{x} > s_k$ or $\coz{y} > s_k$,}\\
\end{cases}
$$
and 
$$
\mathbf{i}^{k,\varepsilon}(x,z) := - \mathbf{i}^{k,\varepsilon}(z,x) := 
\begin{cases}
\mathbf{i}(x,z)  & \hbox{if $\coz{x} \neq s_k$.}\\
\mathbf{i}(x,z) + \varepsilon \mathbf{j}(x,\vec{x}) & \hbox{if $\coz{x} = s_k$ where $\vec{x}$ is the right neighbour of $x$} \\
&\hbox{at level $s_k + 1$ provided the edge $\{x,\vec{x}\}$ exists.}
\end{cases}
$$
It is clear that $\mathbf{i}^{k,\varepsilon}$ satisfies the flow property. In words, the flow $\mathbf{i}^{k,\varepsilon}$ coincides 
with $\mathbf{i}$ for levels below or equal to $s_k$ and coincides with $\mathbf{i} - \varepsilon \mathbf{j}$ for levels above $s_k + 1$. In order to maintain the flow node's law, the missing flow going through the cut-set of horizontal edges between levels $s_k$ and $s_k+1$ is re-routed through the edges at level $s_k$ that link to $z$.  Let $\mathcal{E}(\mathbf{i})$ (resp. $\mathcal{E}(\mathbf{i}^{k,\varepsilon})$) denotes the energy dissipated by $\mathbf{i}$ (resp. $\mathbf{i}^{k,\varepsilon}$). We estimate
\begin{eqnarray}
\nonumber\Delta &:=& \mathcal{E}(\mathbf{i}^{k,\varepsilon}) - \mathcal{E}(\mathbf{i}) \\
\nonumber& = &\frac{1}{2} \sum_{
\stackrel{x,y \in A}{\coz{x} > s_k \hbox{ \tiny{or} } \coz{y} > s_k}
} \Big[\big(\mathbf{i}(x,y) - \varepsilon \mathbf{j}(x,y)\big)^2 - \mathbf{i}(x,y)^2\Big] \\
\label{DeltaEnergy}&&  + \frac{1}{\eta} \sum_{x \in A \;:\; \coz{x} = s_k} 
\Big[\big(\mathbf{i}(x,z) + \varepsilon \mathbf{j}(x,\vec{x})\big)^2 - \mathbf{i}(x,z)^2\Big]
\end{eqnarray}
Here the factor $\frac{1}{\eta}$ corresponds to the resistance of the added edges. Thanks to properties (c) and (d) of the lemma, each term in the first sum is non-positive so we can upper bound this sum by keeping only the terms corresponding to edges of the form $(x,\vec{x})$ where $\coz{x} = s_l$ for some $k \leq \ell \leq d$:   
\begin{eqnarray*}
\frac{1}{2}\sum_{
\stackrel{x,y \in A}{\coz{x} > s_k \hbox{ \tiny{or} } \coz{y} > s_k}
} \Big[\big(\mathbf{i}(x,y) - \varepsilon \mathbf{j}(x,y)\big)^2 - \mathbf{i}(x,y)^2\Big] 
& \leq & \sum_{\ell = k}^d \sum_{\stackrel{x \in A}{\coz{x} = s_\ell}} 
\Big[\big(\mathbf{i}(x,\vec{x}) - \varepsilon \mathbf{j}(x,\vec{x})\big)^2 - \mathbf{i}(x,\vec{x})^2\Big] \\
&=&  -2\varepsilon \sum_{\ell = k}^d \sum_{\stackrel{x\in A}{\coz{x} = s_\ell}} \mathbf{i}(x,\vec{x})\mathbf{j}(x,\vec{x}) + \mathcal O(\varepsilon^2). 
\end{eqnarray*}
Now, since the flow $\mathbf{j}$ has strength $\mathbf{i}(b)$ and there are at most $|\Gamma|$ edges in the cutset of edges linking level $s_\ell$ to $s_\ell + 1$, there must exist some $x \in A$ with $\coz{x} = s_\ell$ such that $|\mathbf{i}(x, \vec{x})| \geq |\mathbf{j}(x, \vec{x})| \geq \frac{\mathbf{i}(b)}{|\Gamma|}$. Thus, we deduce that
\begin{equation}\label{EnergyBound1}
\frac{1}{2}\sum_{
\stackrel{x,y \in A}{\coz{x} > s_k \hbox{ \tiny{or} } \coz{y} > s_k}
} \Big[\big(\mathbf{i}(x,y) - \varepsilon \mathbf{j}(x,y)\big)^2 - \mathbf{i}(x,y)^2\Big] \leq -\frac{2\varepsilon (d -k + 1)\mathbf{i}(b)^2}{|\Gamma|^2} + \mathcal O(\varepsilon^2).
\end{equation}
On the other hand, using the fact that $\mathbf{i}(x,z) \geq 0$ and property (e), the second term in \eqref{DeltaEnergy} can be upper bounded by:
\begin{eqnarray}
\nonumber \frac{1}{\eta} \sum_{x \in A \;:\; \coz{x} = s_k} 
\Big[\big(\mathbf{i}(x,z) + \varepsilon \mathbf{j}(x,\vec{x})\big)^2 - \mathbf{i}(x,z)^2\Big]
& =& \frac{2\varepsilon}{\eta} \sum_{x \in A \;:\; \coz{x} = s_k} \mathbf{i}(x,z) \mathbf{j}(x,\vec{x}) + \mathcal O(\varepsilon^2)\\
\nonumber&\leq & \frac{2\varepsilon \mathbf{i}(b)}{\eta} \sum_{x \in A \;:\; \coz{x} = s_k} \mathbf{i}(x,z) + \mathcal O(\varepsilon^2).\\
\label{EnergyBound2}& = & \frac{2\varepsilon \mathbf{i}(b) \mathbf{i}_k}{\eta}   + \mathcal O(\varepsilon^2). 
\end{eqnarray}
According to Thomson's principle, the unit current has minimal energy among all unit flows hence $\Delta(\varepsilon) \geq 0$ for all $\varepsilon$. Combining \eqref{DeltaEnergy}, \eqref{EnergyBound1} and \eqref{EnergyBound2}, we conclude that
$$
-\frac{2\varepsilon (d -k + 1)\mathbf{i}(b)^2}{|\Gamma|^2} + \frac{2\varepsilon \mathbf{i}(b) \mathbf{i}_k}{\eta} + \mathcal O(\varepsilon^2) \geq 0 
$$  
which finally yields \eqref{ik_j_bound} by letting $\varepsilon$ tend to $0$. 

The remaining of the proof is rather straightforward. First, we sum \eqref{ik_j_bound} for $k = 1,\ldots,d$. Since $\mathbf{i}$ is a unit current, we find that
$$
1 \, = \, \mathbf{i}(b) +  \sum_{k=1}^d \mathbf{i}_k  \,\geq\, \sum_{k=1}^d \mathbf{i}_k \,\geq\, \frac{\eta \mathbf{i}(b)}{|\Gamma|^2 }\sum_{k=1}^d (d-k +1) \,\geq\, \frac{\eta \mathbf{i}(b) d^2}{2 |\Gamma|^2}, 
$$
and therefore, recalling the probabilistic interpretation of $\mathbf{i}(b)$, we have proved that
\begin{equation}\label{LV0}
\bP^{\tilde{A}}_a\big(\hbox{the random walk on $\tilde{A}$ hits $b$ before $z$} \big) \, \leq \, \frac{2|\Gamma|^2}{\eta d^2}.
\end{equation}
Let us now consider the natural coupling of the random walks $X$ (resp. $\tilde{X}$) starting from $a$ on the electrical networks $A$ (resp. $\tilde{A}$) such that both walks coincide until $\tilde{X}$ hits $\{z,b\}$.  More precisely, we construct both walks by first tossing a (biaised) coin at each step  to decide whether $\tilde{X}$ exits by an edge of conductance $\eta$ when the coin gives a "head"  (and such an edge exist) and otherwise move the two walks together. Recall that $\mathcal{L}^r_{\mathcal{S}}$ is the total time spent over the vertices at levels belonging to $\mathcal{S}$ before hitting level $r$ (\emph{i.e.} hitting $b$). Furthermore, the vertices corresponding to levels in 
$\mathcal{S}$ are, by construction, the vertices that share an edge of conductance $\eta$ with $z$. Thus, we have
\begin{equation}\label{LV1}
\bP^A_a\big( \mathcal{L}_{\mathcal{S}} < k\big) \leq \bP^{\tilde{A}}_a\big( \hbox{$\tilde{X}$ hits $b$ before $z$}\big) + \bP\big(\hbox{there is a ``head" in the first $k$ coin throws}\big).
\end{equation}
Note that, each time $\tilde{X}$ is on a vertex that has an edge of conductance $\eta$, there is a probability at most $\frac{\eta}{1+\eta}$ that the associated coin returns "head" (because there is also at least one adjacent edge with unit conductance).  Thus, we get
\begin{equation}\label{LV2}
\bP\big(\hbox{there is a``head" in the first $k$ coin throws}\big) \leq 1 - (1 -\frac{\eta}{1+\eta})^k \leq \eta k.
\end{equation}
Combining \eqref{LV0}, \eqref{LV1} and \eqref{LV2} and choosing $\eta = \frac{\sqrt{2}|\Gamma|}{d\sqrt{k}}$, we conclude that  
$$
\bP^A_a\big( \mathcal{L}^r_{\mathcal{S}} < k\big) \leq \frac{2|\Gamma|^2}{\eta d^2} + \eta k = \frac{\sqrt{8k}|\Gamma|}{d}, 
$$
which completes the proof of the proposition. 
\end{proof}

\begin{proof}[Proof of Lemma \ref{lemmaFlow}]
First, let us notice that we can discard all the edges $e$ on which $\mathbf{i}(e) = 0$, keeping only the connected component of $A$ that contains $\{a,b,z\}$. This is because $\mathbf{j}$ will also be zero on edges where $\mathbf{i}$ is zero. We now assume that $\mathbf{i}$ is non-zero on all edges of $A$. Then, the flow $\mathbf{i}$ induces an oriented graph structure on $A$ so we can speak of ``outgoing'' and ``incoming'' edges from a vertex.  We are going to construct $\mathbf{j}$ starting from $\{b,z\}$ and going backward with respect to the graph orientation. At each step of the exploration process, we keep track of a partition of the vertices into  \emph{inactive}, \emph{active} and \emph{completed} vertices where
\begin{itemize}
\item An \emph{active} vertex has the flow $\mathbf{j}$ defined on some outgoing edges but on no incoming edge.
\item An \emph{inactive} vertex is such that the flow $\mathbf{j}$ is not yet defined on any adjacent edge. 
\item A \emph{completed} vertex is such that the flow $\mathbf{j}$ is already defined on all its adjacent edges. 
\end{itemize}
To begin, we fix $\mathbf{j} = 0$ on all edges adjacent to $z$ and $\mathbf{j} = \mathbf{i}$ on all edges adjacent to $b$. This is possible because there is no edge (with non-zero flow) between $z$ and $b$ thanks to assumption 1. Hence $\mathbf{j}$ satisfies  \emph{(a)} and \emph{(b)}. Now, we set $\{\hbox{\emph{completed}}\} = \{b,z\}$ and $\{ \hbox{\emph{active}} \} = \{ \hbox{neighbours of $b$ and $z$}\}$ while all other vertices are inactive. We show that, at each step,  we can transform an active vertex into a complete one (possibly turning inactive vertices into active ones in the process) while constructing a flow $\mathbf{j}$ which, restricted to the set of completed vertices, still satisfies all the required conditions of the lemma.

Indeed, suppose that we have performed some steps of our exploration process and have our sets of active, inactive and completed vertices.  We claim that there must exist an active vertex $x$ such that $\mathbf{j}$ is already defined on all of its outgoing edges. 
Indeed, if this was not the case, we could start from any active vertex and then recursively construct a path that follows the graph orientation and on which $\mathbf{j}$ is not defined. 
But then, such a path is either infinite or contains cycles. As $A$ is finite, this path has cycles, which contradicts the initial assumption that $\mathbf{i}$ has no cycle. So, let $x$ be such a vertex. The flow $\mathbf{j}$ already defined on the sub-graph spanned by the completed vertices has strength $\mathbf{i}(b)$. In particular, the sum of $\mathbf{j}$ on the  outgoing edges of $x$ is at most $\mathbf{i}(b)$. It is also smaller, by construction, than the sum of $\mathbf{i}$ on the incoming edges of $x$. Thus, it is now clear that we can fix $\mathbf{j}$ on the incoming edges of $x$ in such way that \emph{(c)}, \emph{(d)}, \emph{(e)} hold true. There are several ways to do it. For instance, we can set $\mathbf{j}(z,x) := \alpha \mathbf{i}(z,x)$ for any incoming edge $(z,x)$ where $\alpha$ is the ratio of the total outgoing $\mathbf{j}$-flow over the total incoming $\mathbf{i}$-flow. Finally, we move $x$ to the set of completed vertices and activate all its adjacent currently inactive vertices. This complete the induction step and the proof of the lemma.
\end{proof}
\begin{remark}\emph{
Under the hypotheses of Lemma \ref{lemmaFlow}, the function $\mathbf{i} - \mathbf{j}$ is also a flow which satisfies all the properties \emph{(a)} - \emph{(e)} when exchanging the roles of $b$ and $z$. In particular, any flow $\mathbf{i}$ that satisfies the hypotheses of the lemma can be written as the superposition of two flows $\mathbf{j}_b$ and $\mathbf{j}_z$ such that
\begin{enumerate}
\item[(a)] $\mathbf{i} = \mathbf{j}_b + \mathbf{j}_z$.
\item[(b)] $\mathbf{j}_b$ is a flow from $a$ to $b$ and no flow enters nor exits $z$.
\item[(c)] $\mathbf{j}_z$ is a flow from $a$ to $z$ and no flow enters nor exits $b$.
\item[(d)] The flows $\mathbf{i}$, $\mathbf{j}_b$ and $\mathbf{j}_z$ have the same sign on all edges. 
\end{enumerate}
As already noticed during the proof of the lemma, this decomposition is not, in general, unique. In particular, when $\mathbf{i}$ is a current, the flows $\mathbf{j}_b$ and $\mathbf{j}_z$ need not be currents themselves.}
\end{remark}

Proposition \ref{propDiffusif} shows that the time spent on any $d$ distinct level sets is of order (at least) $d^2$ which is the correct diffusive scaling for the random walk on a sub-graph $A$ of $\Z\times \Gamma$ but it provides only a polynomially decaying upper bound. However, it is not difficult to bootstrap the previous result to get an exponential upper bound which is still homogeneous in $d/\sqrt{k}$. 
\begin{corollary}\label{cor:expodec} Under the assumptions of Proposition \ref{propDiffusif}, we have, for any $k, d\geq 1$, 
\begin{equation*}
\bP^A_a\big( \mathcal{L}^r_{\mathcal{S}} < k\big) \leq  3\exp\Big(-\frac{d}{16|\Gamma|\sqrt{k}}\Big).
\end{equation*}
\end{corollary}
\begin{proof} Let $c := 2 e \sqrt{8} |\Gamma|$ and $\ell := \frac{d}{c\sqrt{k}}$. We split the set  $\{s_1,\ldots,s_d\}$ in $\lfloor \ell\rfloor $ groups, each containing at least  $\lfloor d/\ell \rfloor$ consecutive levels. Thus, in order for the local time $\mathcal{L}^r_{\mathcal{S}}$ to be smaller than $k$, it has to be smaller than this value on each group. Making use of the Markov property at the time of first entrance in each group and applying repeatedly Proposition \ref{propDiffusif}, we find that
$$
\bP^A_a\big( \mathcal{L}^r_{\mathcal{S}} < k\big) \leq  \left(\frac{\sqrt{8k}|\Gamma|}{ \lfloor  d/\ell \rfloor}\right)^{\lfloor \ell\rfloor} \leq \left(\frac{\sqrt{8k}|\Gamma|}{ \frac{c\sqrt{k}}{2}}\right)^{ \ell - 1}
 = e^{-\frac{d}{c\sqrt{k}} + 1} \leq 3 e^{-\frac{d}{16|\Gamma|\sqrt{k}}}.
$$ 
\end{proof}

We can now complete the proof of Proposition \ref{outbound}.

\begin{proof}[Proof of Proposition \ref{outbound}]
Set $c := 4^6 |\Gamma|^6$ which is the constant appearing in Corollary \ref{cor:hitfrontA_new}. Let $\mathcal{S} := \{x_1,\dots,x_d\}$ and define by induction the sequence of stopping time $(T_i)_{i\ge  0}$ by $T_0=0$, and 
$$T_{i+1} := \inf \{n>T_i + c\, :\, \coz{S_n} \in \mathcal{S}\} \quad \hbox{for $i\ge 1$,}$$
with the usual convention that $\inf \emptyset = +\infty$. Let $E_i$ be the event that the walk does not cross any edge of $\partial_eA$ during the time interval $[T_i,T_{i+1}-1]$: 
$$E_i:= \{T_i < \infty \text{ and there does not exist } n\in [T_i,T_{i+1}) \text{ such that } (S_n,S_{n+1}) \in \partial_e A\}.$$
We can couple the random walk on the electrical network \eqref{ORRWnetwork} with the simple random walk on the subgraph $A$ up to time $\sigma$ (the time when the walk on the electrical network leaves $A$). Thus, we deduce that, for any fixed $n$, 
\begin{eqnarray}
\nonumber\bP_a\big( H_r < \sigma \big) &\leq& \bP_a\big( H_r  < \sigma \hbox{ and } H_r < T_{n} \big) + \bP_a\big( T_{n}  \leq \sigma \big)\\
&\leq& \bP^A_a\big( H_r  < T_{n} \big) + \bP_a\left( \bigcap_{i=1}^{n-1} E_i \right). \label{eq:spdelta0}
\end{eqnarray}
On the one hand, before time $T_n$, the simple random walk on $A$ cannot visit levels of $\mathcal{S}$ more than $cn$ times. Therefore, according to Corollary \ref{cor:expodec}, we have
\begin{equation}\label{eq:spdelta1}
\bP^A_a\big( H_{r} < T_n \big) \;\leq\; \bP^A_a\big( \mathcal{L}^r_{\mathcal{S}} < cn \big) \;\leq\; 
3\exp\Big(-\frac{d}{16|\Gamma|\sqrt{c n}}\Big).
\end{equation}
On the other hand. Each time the walk on the electrical network \eqref{ORRWnetwork} visits a site of $\partial_v A$, there is, at least one adjacent exit edge so it has probability at least $\frac{1}{1 + |\Gamma| (\delta + 1)}$ to cross an edge of $\partial_e A$ at the next step. Combining this fact with Corollary \ref{cor:hitfrontA_new} and using the strong Markov property, we deduce that, 
$$
\bP^A_a\big( E_i \,\big|\, E_1,\ldots,E_{i-1} \big) \,\leq\, 1 - \frac{1}{2(1 + |\Gamma| (\delta + 1))} \,\leq \,1 - \frac{1}{4|\Gamma|(\delta + 1)}
$$
which implies, 
\begin{equation}\label{eq:spdelta2}
\bP_a\left( \bigcap_{i=1}^{n-1} E_i \right) \,\leq\, \left(1 - \frac{1}{4|\Gamma|(\delta + 1)}\right)^{n-1} \,\leq\, \exp\left(-\frac{n}{8|\Gamma|(\delta+1)}\right)
\end{equation}
Thus, combining \eqref{eq:spdelta0}, \eqref{eq:spdelta1}, \eqref{eq:spdelta2}, we find that
\begin{equation}\label{eq:spdelta3}
\bP_a\big( H_r < \sigma \big) \,\leq \,  3\exp\Big(-\frac{d}{16|\Gamma|\sqrt{c n}}\Big) +  \exp\left(-\frac{n}{8|\Gamma|(\delta+1)}\right). 
\end{equation}
Finally, setting $n = \lfloor x \rfloor$ with $x = \frac{d^{2/3}(1+\delta)^{2/3}}{8|\Gamma|^2}$ and recalling the exact value of $c$, we conclude that
\begin{eqnarray*}
\bP_a\big( H_r < \sigma \big) &\leq& 3\exp\Big(-\frac{d}{16|\Gamma|\sqrt{c x}}\Big) + 2 \exp\left(-\frac{x}{8|\Gamma|(\delta+1)}\right)\\
&=& 5\exp\left(-\frac{1}{4^4|\Gamma|^3}\left(\frac{d^2}{1+\delta}\right)^{\frac{1}{3}}\right).
\end{eqnarray*}
\end{proof}

\subsection{Auxiliary results}\label{sec.aux}
Proposition $\ref{outbound}$ gives a stretched exponential upper bound for the probability of crossing $d$ levels without exiting a reinforced sub-graph. 
However, this bound is meaningful only for $d \gg \delta^{1/2}$ (\emph{i.e.} when the probability goes to $0$). In the next section, we will need a bound of this same probability in the regime $d \approx \delta^{1/4}$ (when the probability goes to $1$). But one can still use the corollary \ref{cor:expodec}, and adapt the proof of Proposition \ref{outbound} to cover our needs. 
Note that it will be convenient now to 
state the results with a site $a$ on the right of level $0$, and to consider the hitting time of level $0$ instead of level $r$, but of course this is strictly equivalent to the previous formulation.    

\begin{lemma}\label{lem:exitD}
Let $d\ge 1$, and $A$ be a connected subgraph of $\{0,\dots,d\} \times \Gamma$, such that for any $i\in \{0,\dots,d\}$ there is at least one site at level $i$ in $\partial_vA$, and let $a\in A$, with $\overline a = d$. 
Define the random walk $S$ and the exit time $\sigma$ from $A$, as in Proposition \ref{outbound}.  
There exist positive constants $c$ and $C$ (not depending on any parameter), such that 
for any $d\ge C|\Gamma|^8$, and $d^{3/2}\le \delta \le (2d)^5$, 
$$\bP_a[\sigma <H_0] \ \ge\ c\cdot\frac{d^{3/2}}{|\Gamma|^6\delta}.$$ 
\end{lemma}
\begin{proof}
Let  
$$\mathcal L:=|\{0\le n< H_0\wedge \sigma :\ S_n\in \partial_vA\}|,$$ 
be the time spent on $\partial_vA$ before time $H_0\wedge \sigma$. Set $N=d^{3/2}|\Gamma|$ and $L:=N/(4^7|\Gamma|^6)$, and note that the hypothesis on $d$ implies that $L\ge 1$, at least for  $C$ large enough (not depending on any parameter). Then decompose 
\begin{align}\label{eq1}
\bP_a(H_0\le\sigma) \le  \bP_a(H_0\le N \wedge \sigma) + \bP_a(\sigma \wedge H_0>N,\, \mathcal L\le L)+ \bP_a(\mathcal L > L).
\end{align}
Note that the random walk $S$ can be coupled with the simple random walk on $A$ (which by definition never exits $A$), in such a way that they coincide up to the time $\sigma$. Therefore the first probability on the right-hand side of \eqref{eq1} can be bounded using Corollary \ref{cor:expodec}, and we get, for $C$ large enough,
\begin{equation}\label{eq11}
\bP_a(H_0\le N \wedge \sigma) \ \le \ \bP_a^A(H_0\le  N) \ \le \ 3\exp(-\frac{d^{1/4}}{16|\Gamma|^{3/2}}) \ \le\ \frac{1}{6}\cdot\frac{d^{3/2}}{4^8|\Gamma|^6\delta}. 
\end{equation}
Using now, as in the proof of Proposition \ref{outbound}, that each time the process is on a vertex of $\partial_vA$, it has probability at least  $\frac{1}{1+|\Gamma|(1+ \delta)}$ to exit $A$,  we get 
\begin{align}\label{eq2}
\bP_a(\mathcal L > L)\ \le\ \left(1-\frac{1}{1+|\Gamma|(1+ \delta)}\right)^L \ \le\ 
1-\frac{d^{3/2}}{2\cdot 4^7|\Gamma|^5(1+|\Gamma|(1+\delta))}\ \le\  1-\frac{d^{3/2}}{3\cdot 4^7|\Gamma|^6\delta},
\end{align}
using for the second inequality that $(1-\varepsilon)^n \ge 1-n\varepsilon/2$, when $\varepsilon\le 1/(2n)$ together with the hypothesis on $d$ and $\delta$, and using for the last one that $C$ is large enough.
Finally, Corollary \ref{cor:hitfrontA_new} and the fact that $A$ does not contain any level set show that, under the event $\{\sigma\wedge H_0\ge N\}$, $\mathcal L$ is stochastically larger than a Binomial random variable with number of trials $N/(4^6|\Gamma|^6)=4L$, and probability of success $1/2$. 
Then, again Hoeffding's inequality yields, choosing $C$ large enough, 
\begin{align}\label{eq4}
\bP_a[H_0\wedge \sigma\ge N,\, \mathcal L\le L]\  \le\ \exp(-\frac{L}{2})\le\ \exp(-2\frac{d^{3/2}}{4^{8} |\Gamma|^{5}})\le \frac{1}{6}\cdot\frac{d^{3/2}}{4^8|\Gamma|^6\delta}.
\end{align}
The lemma follows from \eqref{eq1}, \eqref{eq11}, \eqref{eq2}, and \eqref{eq4}, choosing $c=4^{-8}$.
\end{proof}

As a consequence we obtain the following result. 

\begin{corollary}\label{lem:newlevel}
Under the setting of Lemma \ref{lem:exitD}, one has for some constant $C'>0$, 
$$\bP_a\left[\coz S_\sigma =d+1 \mid \sigma<H_0 \right] \ \le \  C'\cdot\frac{|\Gamma|^5}{d^{3/4}} .$$
\end{corollary}
\begin{proof}
Denote by $e_\sigma$ the random exit edge $(S_{\sigma-1},S_\sigma)$, and let $F$ be the set of horizontal edges in $\partial_eA$ whose tail is on 
level $d$ and 
head on level $d+1$. Now fix some $f\in F$, and denote by $E(f)$ the set of edges in $\partial_eA$, whose tail is at (intrinsic) distance smaller than $d^{3/4}/|\Gamma|^2$ from $f^-$ in $A$.
We claim that $E(f)$ contains at least $d^{3/4}/(2|\Gamma|^3)$ edges. 
To see this, fix any non intersecting path in $A$ of length $d^{3/4}/|\Gamma|^2$ starting from $f^-$. The fact that any level set contains at least one site in $\partial_vA$, 
and that the diameter in $\Gamma$ is bounded by $|\Gamma|-1$,  
imply that to every $2|\Gamma|$ consecutive vertices in the path one can associate
in a one to one way an edge in $\partial_eA$ whose tail is at distance smaller than $|\Gamma|-1$ from this path, 
and this yields our claim on the size of $E(f)$. Now Proposition \ref{prop:balance} implies that for any $e\in E(f)$, one has
$$| \bP_a(e_\sigma = f) - \bP_a(e_\sigma = e)| \le \frac{d^{3/4}}{|\Gamma|^2|(1+\delta)}.$$
Likewise, by conditioning first on the position of the walk at time $H_0$, we also have 
$$| \bP_a(e_\sigma = f,H_0<\sigma) - \bP_a(e_\sigma = e,H_0<\sigma)| \le \frac{d^{3/4}}{|\Gamma|^2|(1+\delta)}.$$
Therefore, by triangle inequality, 
$$| \bP_a(e_\sigma = f,\sigma<H_0) - \bP_a(e_\sigma = e,\sigma<H_0)| \le 2\frac{d^{3/4}}{|\Gamma|^2|(1+\delta)}.$$ 
Combining this with the result of Lemma \ref{lem:exitD}, we obtain 
$$| \bP_a(e_\sigma = f\mid \sigma<H_0) - \bP_a(e_\sigma = e\mid \sigma<H_0)| \le 2 \frac{|\Gamma|^4}{cd^{3/4}}.$$
As a consequence, one has either $\bP_a(e_\sigma = f\mid \sigma<H_0)\le  (4/c)\cdot |\Gamma|^4 d^{-3/4}$, or 
\begin{align*}
1 & \ge\, \sum_{e\in E(f)}\bP_a(e_\sigma = e\mid \sigma <H_0)\, \ge\,   \frac{d^{3/4}}{2|\Gamma|^3} \left(\bP_a(e_\sigma = f \mid \sigma <H_0) - 2\frac{|\Gamma|^4}{cd^{3/4}}\right)\\ 
& \ge\,  \frac{d^{3/4}}{4|\Gamma|^3}\bP_a(e_\sigma = f\mid \sigma <H_0).
\end{align*}
The result follows since $F$ contains at most $|\Gamma|$ edges. 
\end{proof}

\section{$D$-walls}\label{section:Dwall}
In this section we consider the ORRW on $G=\Z\times \Gamma$, which we recall we denote by $(X_n)_{n\ge 0}$.
We also recall that its range $\RR_n$ (which we will also sometimes write as $\RR(n)$) is the graph consisting of all vertices visited 
up to time $n$ and all edges crossed by the walk up to this time  (in one or the other direction).

We define now the notion of {\bf $D$-wall} created by the ORRW, which extends the definition of walls used by Vervoort in \cite{Ver} 
(the latter being just $1$-walls). 
The reason of this new definition is that we want to ensure the typical spacing between two consecutive walls being of polynomial order in the 
size of $\Gamma$ (instead of exponential order when using the simpler notion of wall from \cite{Ver}).

Given positive integers $x$ and $D$, we say that level $x$ {\bf begins a $D$-wall} if the walk $S$ visits a complete level set on the right of $x$ before reaching level $x + D$:
$$
\hbox{there exists $y\in \{x, x+1,\dots,x + D-1\}$ such that } \{y\}\times \Gamma\ \subset\ \RR(H_{x+D}).
$$
To simplify the discussion below, we use the convention that $x$ also begins a $D$-wall when $H_{x+D}=\infty$. 
The usefulness of $D$-walls will appear more clearly later. The basic idea is that they enable us to control the intrinsic distance within 
the range in order to  match pairs of (non-reinforced) oriented edges with opposite direction that are close together. 
Then Propositions \ref{prop:balance} and \ref{outbound} show that the sum of their contributions to the drift part in the martingale that 
will be defined later in \eqref{theMartingale} is negligible, ensuring recurrence. First, we must check that these walls occur often enough, at least when $D$ is suitably chosen as a function of $\delta$. This is precisely the purpose of the next proposition, which is the main result of this section (let us stress for ease of the reading that $D$ will later be taken to be of order $\delta^{1/4}$). 

\begin{proposition}\label{prop:wall}
There exist a constant $C>0$, such that for any $D\ge C|\Gamma|^{10}$, and $D^{3/2}\le \delta \le D^5$, one has for any $x\ge 0$, almost surely on the event $\{H_x<\infty\}$, 
$$\bP[x\text{ begins a $D$-wall}\mid \FF_{H_x}] \ \ge\ \frac 12.$$
\end{proposition}

We first need an intermediate lemma, which follows from the results proved in Subsection \ref{sec.aux}.  
\begin{lemma}\label{corsigma12}
Consider the ORRW $(X_n)_{n\ge 0}$, starting from some vertex at level smaller than $0$. 
Assume that at some time $n_0<H_{d+1}$, one has $\coz X_{n_0} = d$, and that none of the level sets $i$, with $0\le i\le d$, is contained in $\RR_{n_0}$. Assume further that the hypotheses on $d$, $\delta$ and $|\Gamma|$ from Lemma \ref{lem:exitD} are satisfied. Let 
$$\sigma_1:= \inf\{n\ge n_0 \, : \, (X_{n-1},X_n)\in \partial_e\RR_{n-1},\, 1\le \coz X_{n-1}  \le d, \text{ and }1\le \coz X_n  \le d\},$$
and 
$$\sigma_2:= \inf\{n\ge n_0 \, : \, \coz X_n = d+1\}=H_{d+1}.$$
Then with the same constant $C'>0$ as in Corollary \ref{lem:newlevel}, one has 
$$\bP[\sigma_2<\sigma_1\mid \mathcal F_{n_0}] \ \le\  \frac{C' |\Gamma|^5}{d^{3/4}}.$$
\end{lemma}
\begin{proof}
Define $(H_{0,i})_{i\ge 0}$, and $(H_{d,i})_{i\ge 0}$ recursively by $H_{d,0}=0$, and for $i\ge 0$,
$$H_{0,i}:=\inf\{n\ge H_{d,i} \ : \coz X_n =0\},$$
and 
$$H_{d,i+1}:=\inf\{n\ge H_{0,i} \ : \coz X_n =d\}.$$ 
Then the Markov property and Corollary \ref{lem:newlevel} give 
\begin{align*}
\bP[\sigma_2<\sigma_1\mid \mathcal F_{n_0}] &= \ \sum_{i\ge 0} \bP[H_{d,i}< \sigma_2<\sigma_1\wedge H_{d,i+1}\mid \mathcal F_{n_0}] \\
&= \ \sum_{i\ge 0} \bP[H_{d,i}< \sigma_2\wedge \sigma_1< H_{0,i}, \bar{S}_{\sigma_2\wedge\sigma_1}=d+1\mid \mathcal F_{n_0}] \\
& \le\ \frac{C'|\Gamma|^5}{d^{3/4}}\sum_{i\ge 0} \bP[H_{d,i}< \sigma_2\wedge \sigma_1< H_{0,i}\mid \mathcal F_{n_0}]  \leq \ \frac{C'|\Gamma|^5}{d^{3/4}}.
\end{align*}
\end{proof}

We are now in position to give the proof of Proposition \ref{prop:wall}.

\begin{proof}[Proof of Proposition \ref{prop:wall}]
Denote by $\tau$ the first time when a level set in $\{x,\dots,x+D-1\}\times \Gamma$ is covered, and set for $k\ge 0$, 
$\tau_k:=H_{x+ k+\lfloor D/2\rfloor}$.  
Then for each $k\ge0$, define $(\tau_{k,i})_{i\ge 0}$ and $(\sigma_{k,i})_{i\ge 0}$ by $\tau_{k,0}:=\tau_k$, and for any $i\ge0$,
\begin{align*}
\sigma_{k,i}&:=\inf\{t > \tau_{k,i}\ :\ (X_{t-1},X_t)\in \partial_e\RR_{t-1}\text{ and } x+1\le  \coz{X}_{t-1}\le x+k+\lfloor D/2\rfloor\},\\
\tau_{k,i+1}&:= \inf\{\sigma_{k,i}\le t<  \tau_{k+1}\ :\ \coz{X}_t= x+k+\lfloor D/2\rfloor \}.
\end{align*}
Note that $\{\tau_{k,i+1}=\infty\} \cap \{\tau_{k+1}<\infty\}= \{\sigma_{k,i} = \tau_{k+1}<\infty\}$. Finally, define 
\[
N(k)=\inf\{i\ge1:\ \tau_{k,i}=\infty\}=\sum_{i=0}^\infty \1{\tau_{k,i}<\infty}.
\]
For each $k\ge0$, the random variable $N(k)$ counts how many times $X$ starts from $k$, and crosses a new edge with tail at a level between $x+1$ and $x+k+\lfloor D/2\rfloor$, before $\tau_{k+1}$.

Now, letting $d=\lfloor D/2\rfloor$, Lemma \ref{corsigma12} implies that for any $k,i\ge 0$, almost surely, 
\begin{align*}
\bP[\tau_{k,i+1}=\infty\mid \FF_{\tau_{k,i}}] {\bf 1}\{\tau_{k,i}<\tau\}\ \le\  \frac{C'|\Gamma|^5}{d^{3/4}}\ \le\ \frac{1}{10|\Gamma|^2}, 
\end{align*}
using the hypothesis that $D\ge C|\Gamma|^{10}$, and choosing $C$ large enough, for the last inequality. 
Hence, if no level set is covered by time $\tau_{k+1}$, $N(k)$
stochastically dominates a geometric random variable with parameter $1/(10|\Gamma|^2)$. 
On the other hand, by definition, the number of edges in $\{x,\dots,x+D-1\}\times \Gamma$, which are crossed before time $H_{x+D}$ 
is larger $\Sigma:=N(1)+\dots+ N(\lfloor D/2\rfloor)$. The above discussion shows that on the event $\{\tau>H_{x+D}\}$, this sum stochastically dominates the sum of $\lfloor D/2\rfloor$, i.i.d. Geometric random variables with parameter $1/(10|\Gamma|^2)$.  
Thus it follows from Bernstein's inequality that with probability at least $1/2$, $\Sigma$ is larger than $5|\Gamma|^2\lfloor D/2\rfloor$. However, the latter quantity is always larger than the total number of edges on $\{x,\dots,x+D\}\times \Gamma$, leading to a contradiction. We conclude that with probability at least $1/2$, one has $\tau<H_D$, as wanted. 
\end{proof}
 
\begin{remark}\emph{
Proposition \ref{prop:wall} shows the existence of $D$-walls at typical distance of order a constant, when $D$ is a power of $\delta$, and 
$\delta$ is  polynomially large with respect to $|\Gamma|$. However, if one does not care about polynomial bounds, a much simpler argument gives a weaker exponential bound. Indeed it is easy to check that with a probability of order $\exp(-|\Gamma|^3)$, a level set is entirely covered before the next level is discovered, in other words that a $1$-wall begins.}
\end{remark}

\section{Gambler's ruin type estimates}\label{sec:gambler}
\subsection{A martingale}
We consider now $\{X_n\}_{n\ge 0}$ the ORRW on $\Z\times \Gamma$ starting from some vertex at level $0$. 
We then define the process 
\begin{equation}\label{theMartingale}
M_n\ :=\ \coz X_n + \delta\cdot  \sum_{k=0}^{n-1} (\coz{X}_{k+1}-\coz X_k) \, \1{\{X_k,X_{k+1}\} \notin E_k},
\end{equation}
where $E_k$, defined in \eqref{defEn}, is the set of edges crossed by time $k$.
The following fact was observed by Vervoort \cite{Ver}\footnote{similar martingales related to other reinforcement schemes were also previously constructed by Davis \cite{Dav90}.}.
\begin{lemma}\label{lem:martingale}
The process $(M_n)_{n\ge 0}$ is an $(\FF_n)_{n\ge 0}$-martingale. 
\end{lemma}
\begin{proof}
Assume that, at time $n\in\mathbb{N}$, we have $X_n=s=(x,\gamma)\in\mathbb{Z}\times\Gamma$. 
If $S_{n+1}\in \{x\}\times \Gamma$, then $M_{n+1}=M_n$. Thus, to compute the conditional expectation of $M_{n+1}-M_n$, we only need to consider the cases when $X_{n+1}=(x-1,\gamma)$ or $X_{n+1}=(x+1,\gamma)$. Denoting $e_1=\{s,(x+1,\gamma)\}$ and $e_2=\{s,(x-1,\gamma)\}$, we obtain
\begin{align*}
\E\left[\left.M_{n+1}-M_n\right|\mathcal{F}_n\right]&=\frac{(1+\delta\1{e_1\in E_n})(1+\delta\1{e_1\notin E_n})+ (1+\delta\1{e_2\in E_n})(-1-\delta\1{e_2\notin E_n})}{\displaystyle{\sum_{z:z\sim s}\left(\delta\1{\{s,z\}\in E_n}+1\right)}}\\
&=0.
\end{align*}
Therefore, $(M_n)_n$ is a martingale.
\end{proof}

\subsection{Gambler's ruin estimates}
Here we prove gambler's ruin type estimates for the (horizontal coordinate of the) ORRW on $\Z\times \Gamma$. 
These will be our main tools for the proofs of Theorems  
\ref{theo1} and \ref{theo2}. We first present a relatively simple form of the result, Proposition \ref{prop:main}, which will be sufficient for the proof of Theorem \ref{theo1} and Proposition \ref{returntime}. Then we will give a slightly more precise and more effective version, Proposition \ref{prop:main2}, which is needed for the proof of Theorem \ref{theo2}.  But first let us introduce some new notation. For $x\ge 0$, let 
$$H_{x,0}:=\inf\{n\ge H_x\ :\ \coz X_n=0\},$$
be the first return time to level $0$ after time $H_x$.  Here is our first result: 

\begin{proposition}\label{prop:main}
There exists a constant $C>0$, such that for any $\delta\ge C|\Gamma|^{40}$, almost surely for any $x$ large enough, on the event $\{H_x<\infty\}$,  
$$\bP[H_{2x} < H_{x,0}\mid \FF_{H_x}] \ \le\ \frac 1{2^{10}}.$$
\end{proposition}
\begin{remark}\emph{The exact value of the constant in the upper bound does not have any serious meaning, the only reason for the choice of $2^{-10}$ is for convenience for the  proof of Proposition \ref{returntime}. }
\end{remark}
\begin{proof}
Fix some $x>0$, and assume that $H_x$ is finite. Then consider $(\sigma_n)_{n\ge 0}$, the successive times of visit of new edges after $H_x$. Formally 
 $\sigma_0=H_x$, and for $n\ge 1$, 
$$\sigma_n:=\inf\{k> \sigma_{n-1}\ :\ \{X_{k-1},X_k\} \notin \RR_{k-1} \}.$$
Then set $A_n:=\RR_{\sigma_n}$, which we recall we consider as a subgraph of $G$, where the edges are those traversed by the walk, and let
$$D_n:= \sum_{k=1}^n (\coz X_{\sigma_k}-\coz X_{\sigma_k-1}) {\bf 1}(\sigma_k\le n),$$
be the total drift accumulated after $H_x$ and before time $n$. For simplicity write also $\tau_x:=H_{x,0}\wedge H_{2x}$. 
Applying the optional stopping time theorem, and using Lemma \ref{lem:martingale}, we get 
\begin{equation}\label{eq:stoptheo}
2x\cdot \bP[H_{2x} < H_{x,0}\mid \FF_{H_x}]+ \delta\, \E[D_{\tau_x}\mid \FF_{H_x}] \ =\ x.
\end{equation}
Thus it amounts to estimate the expected value of the total drift accumulated at time $\tau_x$.

To this end, we will bound the expected drift at each visit of a new edge. 
When both the tail and the head of the edge are on the same level, then the drift is zero, so nontrivial contributions only 
come from horizontal edges.
Informally we need to show that the walk has at least as much chance (to the leading order) to exit $A_n$ 
through a directed edge $e=(v,w)$ oriented positively, i.e. such that $\coz{w}=\coz{v}+1$, as through an edge oriented in the other direction. 
This holds actually only up to small error terms which are of sub-linear order in $\delta$. 
Fortunately the latter are compensated by a strong drift equal to $\delta$ accumulated each time the process arrives at a new level for the first time, which on the event $\{H_{2x}<H_{x,0}\}$ happens exactly $x$ times.

To be more precise, from now on, and in the whole section, we fix $D=\lfloor \delta^{1/4}\rfloor$.   
Since we assume $\delta\ge C |\Gamma|^{40}$, we have also $D\ge C^{1/4} |\Gamma|^{10}$, so that by taking large enough $C$ one can apply the result of Proposition \ref{prop:wall}.

Now let $z$ be the smallest positive integer such that $z$ begins a $D$-wall,  
with the notation of Section \ref{section:Dwall}, and let $D_0:=z+D$. 
Since $D_0$ is almost surely finite (as a consequence of Proposition \ref{prop:wall}), 
we can  assume that $x$ is larger than $D_0$, by taking larger $x$ if necessary.

Next, for $n\ge 0$, let 
$$\EE^+_n=\left\{e=(v,w)\in \partial_e A_n\ :\ \coz w = \coz v +1\text{ and }\coz{v}\in \{D_0,\dots,2x-2\}\right\},$$ 
and 
$$\EE_n^-:=\left\{e=(v,w)\in \partial_e A_n\ :\ \coz w = \coz v -1 \text{ and }\coz{v} \in \{D_0,\dots,2x-1\}\right\}.$$ 
Then we define an injective map $\varphi_n: \EE_n^-\to  \EE^+_n$ as follows. 
If $e\in \EE_n^-$, with $e^-=(y,\gamma)$, 
we let $\varphi_n(e)$ be the edge in $\EE_n^+$, such that $ \varphi_n(e)^-=(y',\gamma)$, with the largest possible $y'$ smaller 
than or equal to $y$. Such an edge necessarily exists as soon as there exists a $D$-wall on the left of $e$, which we assumed. 
Moreover, if $z$ is the largest integer beginning a $D$-wall with $z+D\le y$, and $z'$ is the smallest integer beginning a $D$-wall, which is  larger than $y$, then one has $d_{A_n}(e^-,\varphi_n(e)^-)\le |\Gamma|(z'+D-z)$, using again that the diameter in $\Gamma$ is bounded by $|\Gamma|- 1$. 
Note that a $D$-wall on the right of $y$ might not have 
been yet discovered at time $\sigma_n$. In this case one can just define $z'$ as the 
maximal level reached at time $\sigma_n$, and the same bound for $d_{A_n}(e^-,\varphi_n(e)^-)$ holds. 
We also set $D_n(e):=(z'+D-z)$, with the above notation. 
Then we define a new intrinsic distance between vertices: given some subgraph $A\subset G=\Z\times \Gamma$, and for any $v,w\in A$, we let 
$$\tilde d_A(v,w):=\# \{i\, :\, \coz v\wedge \coz w \le i \le \coz v\vee \coz w, \text{ and level $i$ is not contained in $A$}\}.$$
Next we define 
$$\GG_n\ :=\ \EE^-_n\cap \{e\, : \, \tilde d_{A_n}(e^-,X_{\sigma_n})\le \delta^\beta \text{ and }D_n(e) \le \delta^\alpha\},$$
with $\alpha$ and $\beta$ some constants satisfying $1/4<\alpha<1/2<\beta<1$, and $\alpha+\beta\le7/8$. 
We also denote by 
$e_n:=(X_{\sigma_n-1},X_{\sigma_n})$, the $n$-th visited (directed) edge after time $H_x$. 
Then set for $n\ge 0$, 
$$I_n:= {\bf 1}(\sigma_n=H_{\overline X_{\sigma_n}}),$$ 
the indicator function that a new level is discovered at time $\sigma_n$.  
One has, on the event $\{\sigma_n<\tau_x\}$, 
\begin{eqnarray}\label{drift1}
\nonumber && \E\left[(D_{\sigma_{n+1}}-D_{\sigma_n}){\bf 1}(\{\sigma_{n+1}\le\tau_x\})  \mid \FF_{\sigma_n}\right] \\
\nonumber &\ge &  \bP\left[I_{n+1}=1,\, \sigma_{n+1}\le \tau_x \mid \FF_{\sigma_n}\right]  - \bP\left[\coz e_{n+1}^+=\coz e_{n+1}^- -1 \in  \{0,\dots,D_0-1\}\mid \FF_{\sigma_n}\right]\\ 
\nonumber && - \bP[D_n(e_{n+1})>\delta^\alpha \mid \FF_{\sigma_n}] - \bP[\tilde d_{A_n}(e_{n+1}^-,S_{\sigma_n})> \delta^\beta \mid \FF_{\sigma_n}] \\
&&-  \sum_{e\in \GG_n} \left(\bP\left[e_{n+1}=e \mid \FF_{\sigma_n}\right]-\bP\left[e_{n+1}=\varphi_n(e) \mid \FF_{\sigma_n}\right]\right). 
\end{eqnarray}
The last sum above can be bounded using Proposition \ref{prop:balance}, which gives 
\begin{equation*}
\sum_{e\in \GG_n} |\bP\left[e_{n+1}=e \mid \FF_{\sigma_n}\right]-\bP\left[e_{n+1}=\varphi_n(e) \mid \FF_{\sigma_n}\right]| \ \le\   2|\Gamma|^2\cdot \delta^{\alpha+\beta -1}, 
\end{equation*}
since the number of edges in $\GG_n$ is bounded by $2|\Gamma|\delta^{\beta}$, and for any $e\in\GG_n$, one has 
by the above discussion $d_{A_n}(e^-,\varphi_n(e)^-)\le |\Gamma|\delta^\alpha$. Moreover, the total number of edges between levels $1$ and $2x$ is bounded by $2x|\Gamma|^2$. Thus 
\begin{equation}\label{drift2}
\sum_{n\ge 0} \sum_{e\in \GG_n} |\bP\left[e_{n+1}=e \mid \FF_{\sigma_n}\right]-\bP\left[e_{n+1}=\varphi_n(e) \mid \FF_{\sigma_n}\right] |{\bf 1}(\sigma_n<\tau_x)\ \le\  4x|\Gamma|^4\,  \delta^{\alpha+\beta -1}.  
\end{equation}
Likewise the number of edges oriented negatively between levels $0$ and $x_0$ is bounded by $|\Gamma|x_0$. Thus 
\begin{equation}\label{drift3}
\sum_{n\ge 0} \bP\left[\coz e_{n+1}^+=\coz e_{n+1}^- -1 \in  \{0,\dots,D_0-1\}\mid \FF_{\sigma_n}\right] {\bf 1}(\sigma_n<\tau_x) \ \le\  |\Gamma| D_0.
\end{equation}
Using Proposition \ref{outbound}, and the hypothesis $\delta\ge C |\Gamma|^{40}$, we obtain similarly by taking larger $C$ if necessary,  
\begin{equation*}
\sum_{n\ge 0} \bP[\tilde d_{A_n}(e_{n+1}^-,S_{\sigma_n}) >\delta^{\beta} \mid \FF_{\sigma_n}]{\bf 1}(\sigma_n<\tau_x) \ \le\ x\cdot \delta^{-1}. 
\end{equation*}
For the last remaining term in \eqref{drift1}, let us denote by $\LL_x$ and $\LL_{x,2x}$ the sum of the distances between two consecutive $D$-walls whose distance exceeds $\delta^\alpha$, respectively between levels $0$ and $x$ and between levels $x$ and $2x$. First one has 
\begin{equation}\label{drift4}
 \E\Big[\sum_{n\ge 0} \bP[D_n(e_{n+1})>\delta^\alpha \mid \FF_{\sigma_n}]{\bf 1}(\sigma_n<\tau_x)\, \Big| \,\FF_{H_x}\Big] \ \le\ |\Gamma|^2\left(\LL_x+\E[\LL_{x,2x}\mid \FF_{H_x}]\right).
 \end{equation}
Moreover, Proposition \ref{prop:wall} shows that the distance between two consecutive $D$-walls is stochastically dominated by a geometric random variable with parameter $1/2$, times $D$. Using that geometric random variables have exponential tail, we get that almost surely, for all $x$ large enough, 
\begin{equation}\label{drift5}
\LL_x \le x\cdot \delta^{-1}\quad \text{and}\quad \E[\LL_{x,2x}\mid \FF_{H_x}]\le x\cdot \delta^{-1}. 
\end{equation}  
On the other hand 
\begin{equation}\label{drift6}
\sum_{n\ge 0} \bP[I_{n+1}=1,\, \sigma_{n+1}\le \tau_x \mid \FF_{H_x}] \ \ge\ x\cdot \bP[H_{2x}<H_{x,0}\mid \FF_{H_x}],  
\end{equation}
simply because on the event $\{H_{2x}<H_{x,0}\}$ the walk arrives exactly $x$ times at a new level for the first time.

Therefore, it follows from \eqref{drift1}, \eqref{drift2}, \eqref{drift3}, \eqref{drift4}, \eqref{drift5}, and \eqref{drift6} that almost surely for all $x$ large enough, 
\begin{eqnarray*}
\E[D_{\tau_x}\mid \FF_{H_x}] \ \ge\  x\cdot \bP[H_{2x}<H_{x,0}\mid \FF_{H_x}] -  8x\, |\Gamma|^4 \, \delta^{\alpha+\beta-1}-|\Gamma|D_0. 
\end{eqnarray*}
Together with \eqref{eq:stoptheo}, it implies 
$$\bP[H_{2x}<H_{x,0}\mid \FF_{H_x}] \ \le\ \frac{1}{2+\delta} + 8 |\Gamma|^4 \, \delta^{\alpha+\beta-1}+\frac{|\Gamma| D_0}{x}.$$
Then Proposition \ref{prop:main} follows, since $\alpha+\beta\le7/8$
and $\delta\ge C|\Gamma|^{40}$, and thus by choosing $C$ large enough, one can make the sum of the two first terms on the righthand side smaller than $2^{-11}$, and similarly for the last term by choosing $x$ large enough.  
\end{proof}

Now we present a stronger form of Proposition \ref{prop:main}, which requires some more notation. For $x\ge 0$, we let $\LL_x$ be the sum of the distances between two consecutive $D$-walls on the left of $x$, whose distance exceeds $\delta^\alpha$, with $\alpha \in (1/4,1/3) $.  Recall also that we fixed the value of $D$ as $D=\lfloor \delta^{1/4}\rfloor$. 
The next lemma will be needed. 
\begin{lemma}\label{exn}
Consider the event $\mathcal E_x:=\{\LL_x\le \frac x\delta\}$. Then for $\delta$ large enough, one has for all $x\ge 1$,  
$$
\bP(\mathcal E_x) \ge 1 - \exp(- \frac x{2\delta^{5/4}}).
$$
\end{lemma}
\begin{proof}
Recall that Proposition \ref{prop:wall} shows that the distance between two consecutive $D$-walls is stochastically dominated by $D\cdot g$, with $g$ a geometric random variable with parameter $1/2$. Moreover, between levels $0$ and $x$, there are certainly less than $x$ integers beginning a $D$-wall. 
Therefore $\LL_x$ is stochastically bounded by the sum of $x$ independent random variables distributed as $D\cdot g{\bf 1}{(g\ge \delta^{\alpha-\frac 14}})$. 
Then the result follows from Chebyshev's exponential inequality. 
\end{proof}
We recall also some notation from the proof of Proposition \ref{prop:main}. 
If $z$ is the first positive integer beginning a $D$-wall, we define $D_0:=z+D$. Note that for any $x\ge 0$, the event $\{x\ge D_0\delta\}$ is $\FF_{H_x}$-measurable. 
Now one can state the refined gambler's ruin estimate necessary to prove Theorem \ref{theo2}. It will be assumed here and in the rest of the paper, that integer parts have to be taken when needed (in particular in the statement of the proposition below $(1+\varepsilon)x$ should be understood as its integer part). 

\begin{proposition}\label{prop:main2}
Let $\alpha\in(1/4,1/2)$ and $\beta \in(1/2,1)$ be fixed. There exists a constant $C>0$, such that for any $\delta\ge C|\Gamma|^{40}$, and any $\varepsilon \in (0,1)$,  
almost surely on the event $\{H_x<\infty\}\cap \{x\ge D_0\delta \}\cap \mathcal E_x$, 
$$\bP[H_{(1+\varepsilon)x} < H_{x,0}\mid \FF_{H_x}] \ \le\ \frac{10|\Gamma|^4\delta^{\alpha + \beta}}{1+\varepsilon \delta}.$$ 
\end{proposition}
We omit the proof of this proposition, since it follows directly from the proof of Proposition \ref{prop:main}.

\section{Proofs of Theorems \ref{theo1} and \ref{theo2}}\label{sec:prooftheo}
We start by proving Theorem \ref{theo1}. 

\begin{proof}[{\bf Proof of Theorem \ref{theo1}}] Proposition \ref{prop:main} shows that, almost surely, each time the process arrives at a new level of the form $2^k$ (resp.$-2^k$), for $k$ large enough, the probability that it goes back to the origin before arriving at level $2^{k+1}$ (resp. $-2^{k+1}$) is lower bounded by a positive constant, independently of the past. Thus the largest level set visited before returning to the origin is stochastically bounded by the exponential of a geometric random variable, and in particular is almost surely finite. 
Since the ORRW cannot stay confined forever in any finite subgraph, this shows that it returns almost surely infinitely many times to level $0$, concluding the proof. 
\end{proof} 

We now turn to the proof of Theorem \ref{theo2}.

\begin{proof}[{\bf Proof of Theorem \ref{theo2}}] 
Define 
$$H^*_n:= \inf\{k\ge 0\ : \  \{n-1,n\}\times \Gamma\subseteq \RR_k\}.$$
It amounts to prove that almost surely for all $n$ large enough, one has 
\begin{equation}\label{goalthm2}
H^*_n< H_{n+n^{\delta^{-1/8}}},
\end{equation}  
We use Proposition \ref{prop:main2}, with $\varepsilon:=1/\delta^\eta$, and $\eta=\frac 18 + \frac 1{100}$.  
We next fix $\alpha$ and $\beta$ such that $\alpha+\beta \le \frac{3}{4} + \frac 1{100}$. Then Proposition \ref{prop:main2} shows that for $\delta\ge C|\Gamma|^{40}$, with $C$ some large constant, 
one has almost surely on the event $\{H_x<\infty\}\cap \{x\ge D_0\delta\}\cap \mathcal E_x$,
\begin{equation}\label{prop:upgrad}
\bP[H_{(1+\varepsilon)x} < H_{x,0}\mid \FF_{H_x}] \ \le\ \frac{10|\Gamma|^4}{\delta^{\frac 18 -\frac 1{50}}} \ \le \ \frac 12,
\end{equation}
taking larger $C$ if necessary for the last inequality. 
Notice moreover, that the environment on the left of the starting position at time $0$ plays no role in the proof of the proposition: it could be anything this would not change the argument nor any of the constants appearing there. 
Therefore, taking the origin of the space to be the position of the walk at time $H_n$, the time origin to be $H_n$, 
and then applying \eqref{prop:upgrad}, 
one deduces that for each $n$, almost surely, on the event 
$\{H_{n+x}<\infty\}\cap \{x\ge \delta D_0(n)\} \cap \mathcal E_x(n)$, 
\begin{equation}\label{prop:upgrad2}
\bP[H_{n+(1+\varepsilon)x} < H_{n+x,n}\mid \FF_{H_{n+x}}] \ \le\ \frac 12, 
\end{equation}
with $D_0(n)$ and $\mathcal E_x(n)$ defined analogously as $D_0$ and $\mathcal E_x$, but concerning the $D$-walls between levels $n$ and $n+x$, and 
$$H_{n+x,n}:=\inf\{t>H_{n+x}\ :\ \coz X_t=n\}.$$
Now define $z_0:=n^{1/\delta^\eta}$, and for $i\ge 1$, set $z_i:=(1+\varepsilon)^i z_0$. Let also 
$$N:= \frac{9}{\delta^\eta\log (1+\varepsilon)}\log n,$$ 
and  
$$\widetilde{\mathcal E}(n)= \bigcap_{i=0}^{N}\ \mathcal E_{z_i}(n).$$  
In particular $z_N\le n^{10/\delta^\eta}$. 
Then define for $i\ge 0$,
$$\varepsilon_i:= {\bf 1}\{H_{n+z_i,n}<H_{n+z_{i+1}}\},$$
and for $t\ge 0$, 
$$\tau_n(t):=\inf\left\{k\ge 0\ :\ \sum_{i=0}^k \varepsilon_i {\bf 1}\{H_{n+z_i}<\infty\} \ge t \right\}.$$
Next introduce the martingale $(M_k)_{k\ge 0}$ (with respect to the filtration $(\FF_{H_{n+z_k}})_{k\ge 0}$), defined by 
$$M_k:= \sum_{i=0}^k \left(\varepsilon_i - \bP(\varepsilon_i=1\mid \FF_{H_{n+z_i}})\right){\bf 1}\{H_{n+z_i}<\infty\}.$$
Equation \eqref{prop:upgrad2} implies that for any $i=0,\dots,N$, almost surely on the event $\mathcal E_{z_i}(n)\cap\{H_{n+z_i}<\infty\}\cap \{z_0\ge D_0(n)\delta\}$, 
$$\bP(\varepsilon_i=1\mid \FF_{H_{n+z_i}}) \ge \frac 12.$$
Using also that $\log (1+\varepsilon)\le \varepsilon$, we obtain $N\ge 9\log n$ (recalling that $\varepsilon = 1/\delta^\eta$). 
Therefore one can fix a small constant $\kappa>0$, such that, at least for $\delta$ large enough, one has $(N-2\kappa \log n)^2/(8N)\ge \alpha \log n$, for some $\alpha>1$.  
Then Azuma's inequality gives that almost surely on the event $\{z_0\ge D_0(n)\delta\}$,  
\begin{align*}
\bP\left(\tau_n(\kappa \log n)>N,\widetilde{\mathcal E}(n), H_{n+z_N}<\infty\mid \FF_{H_{n+z_0}}\right) & \le\  \bP\left(M_N\le \kappa\log n-\frac N2\right)\\
&\le \ e^{-\frac{(N-2\kappa \log n)^2}{8N}}\\
&\le  \ n^{-\alpha}.  
\end{align*}
Recall now Lemma \ref{exn} and that the distribution of $D_0(n)$ has an exponential tail. Together with the above estimate, and Borel-Cantelli's lemma, they imply that almost surely 
$z_0\ge D_0(n)\delta$, for all $n$ large enough, and  
\begin{equation}
\label{thm2.step1}
\text{almost surely for all $n$ large enough, if $H_{n+n^{10/\delta^\eta}}$ is finite, then $\tau_n(\kappa \log n)\le N$.}
\end{equation} 
In other words, almost surely for $n$ large enough the walk makes at least $\kappa \log n$ returns from level $n+ n^{1/\delta^\eta}$ 
to level $n$, before hitting level $n+n^{10/\delta^\eta}$.

At each of these returns, and for every $k\in \{0,\dots,K\}$, with $K:=[n^{1/\delta^\eta}/2]$, 
it has some positive probability independently of the past, say at least $p$ (with $p$ a positive constant depending only on $\delta$ and $|\Gamma|$)  
to cover entirely $\{n+2k-1,n+2k\}\times \Gamma$,  
before exiting this subgraph. Moreover, this holds independently for any $k$, and for every excursion. 
Now to simplify the discussion below,  
given some time $T>0$, let us say that $k$ is $T$-good when $\{n+2k-1,n+2k\}\times \Gamma$, is covered before time $T$, and otherwise that it is $T$-bad. Write also $z_n(i):=z_{\tau_n(i)}$, for $i\ge 0$. 
Note that, using \eqref{thm2.step1}, all that remains to be done to prove \eqref{goalthm2} is to show that 
\begin{equation}
\label{goalbisthm2}
\text{almost surely, $0$ is $H_{n+z_n(\kappa \log n)}$-good, for all $n$ large enough,}  
\end{equation}
since, at least for $\delta$ large enough, one has $10/\delta^\eta \le 1/\delta^{1/8}$, by definition of $\eta$.  
The discussion above shows that if 
$$T_0:=H_{n+z_n(\frac \kappa 2 \log n)},$$ 
then any $k\in \{0,\dots,K\}$, is $T_0$-good 
with probability at least  
$$p_n:=1-(1-p)^{\frac\kappa 2\log n}\ge 1-n^{-\frac{\kappa p}{2}}.$$
Let $K':=[n^{\kappa p/4}]$, and $M:=8/(\kappa p)$. Since the number of integers $k\le K'$ which are $T_0$-bad is stochastically bounded by a Binomial random variable with parameters $K'$ and $K'^{-2}$, standard estimates show that 
$$\bP(\#\{k\le K'\, :\, k\text{ is $T_0$-bad}\}\ge M)\le n^{-2}.$$
Therefore almost surely, for all $n$ large enough the number of integers $k\le K'$ which are $T_0$-bad is bounded by $M$. 
Now for $k\in \{0,\dots,K'\}$, set 
$$d(k):=\inf \{\ell \ge 0 \, :\, k+\ell \text{ is $T_0$-bad}\}.$$
When the number of $T_0$-bad integers is smaller than $M$, for at least one of them $d(k)\ge K'/M$. 
Denote by $k_1$ the smallest of them. Then let $M'$ be the number of $T_0$-bad integers smaller than or equal to $k_1$, and denote them by $k_{M'}<\dots<k_1$. Note that one can assume that $k_{M'}=0$, as otherwise there is nothing more to prove.

Let 
$$T_1:=H_{n+z_n(\frac \kappa 2(1+\frac 1{M'}) \log n)}.$$
Note that between times $T_0$ and $T_1$ there are at least $(\kappa /2M')\log n$ excursions between levels $n+n^{1/\delta^\eta}$ and $n$. 
Note moreover, that if $d(k_1)\ge K'/M$, then between levels $n+2k_1$ and $n+2k_1+N'/M$, the horizontal coordinate of the ORRW behaves as a simple random walk on $\Z$. Now we recall some basic fact on the local time at the origin for this process. If we denote by $L_m$ the time spent at the origin, when the simple random walk (starting from the origin) first hits $\pm m$, then $L_m$ is distributed as a Geometric random variable with parameter $1/m$ (since whenever the random walk is in $\pm 1$, gambler's ruin estimate shows that it has probability exactly $1/m$ to hit $\pm m$ before the origin). In particular  
$$\bP(L_m\le \sqrt m)\le \frac{c}{\sqrt m},$$
for some constant $c>0$. As a consequence for each of the excursions between times $T_0$ and $T_1$, the number of returns to level $n+2k_1$ 
(after a jump to level $n+2k_1+1$) is larger than $\sqrt{K'/M}$ with probability at least $q:=1-c\sqrt{M/K'}$. Therefore in total the number of returns to level $n+2k_1$ between times $T_0$ and $T_1$  
stochastically dominates $\sqrt{K'/M}\cdot B$, with $B=\mathcal B(\frac{\kappa}{2M'}\log n,q)$ (a Binomial random variable with parameters $(\kappa/2M')\log n$ and $q$). One has 
$$\bP(B=0) = (1-q)^{\frac{\kappa}{2M'}\log n} \le e^{-c'(\log n)^2},$$
for some constant $c'>0$. Thus almost surely for $n$ large enough, $B\ge 1$, and there are at least $\sqrt {K'/M}$ returns to level $n+2k_1$ between times $T_0$ and $T_1$. Using then again that at each of these returns (and independently for each of them) the probability to cover $\{n+2k_1-1,n+2k_1\}\times \Gamma$ is larger than $p$, we deduce that 
\begin{equation*}
\text{almost surely, for $n$ large enough $k_1$ is $T_1$-good.}  
\end{equation*}
Then one can just iterate this argument. More precisely, for $i=2,\dots,M'$, one defines 
$$T_i:= H_{n+z_n(\frac \kappa 2(1+\frac i{M'}) \log n)},$$
and using repeatedly the above argument, one shows inductively that 
\begin{equation*}
\text{almost surely, for $n$ large enough, $k_i$ is $T_i$-good, for all $i=1,\dots,M'$,}  
\end{equation*}
proving well \eqref{goalbisthm2}, since we recall that $k_{M'}=0$. 
This concludes the proof of Theorem \ref{theo2}.  
\end{proof}

\section{Finiteness of the expected return times}\label{sec:returntimes}
We prove here that the successive return times to the origin (after an excursion on the right of it) have finite expectation. 
To be more precise, define $(\tau_i)_{i\ge 1}$ and $(\tau_i^+)_{i\ge 0}$ by $\tau_0=0$ and for $i\ge 0$,
$$\tau_i^+ := \inf\{t>\tau_i \ :\ \overline X_t>0\},\quad \text{and}\quad  \tau_{i+1}:=\inf\{t>\tau_i^+ \ :\ \overline X_t = 0 \}.$$

\begin{proposition}\label{returntime}
There exists a constant $C>0$, such that for any $\delta\ge C|\Gamma|^{40}$, one has 
$\E[\tau_i] <\infty$, for all $i\ge 1$. 
\end{proposition}
\begin{remark}
\emph{Recall that the  one-dimensional simple random walk is null recurrent, meaning that its successive return times to the origin have infinite expectation.  
Here the situation is intermediate between positive and null recurrence, since with   
Theorem \ref{theo2} one can see that $\E[\tau_{i+1}-\tau_i]$ diverges when $i$ goes to infinity.}
\end{remark}
\begin{proof}[Proof of Proposition \ref{returntime}]
First note that by symmetry, it suffices to show that for all $i\ge 0$, $\E[\tau_{i+1}- \tau_i^+]$ is finite. 
Write $M_i:=\sup\{\overline X_t \ : \ t\le \tau_i\}$, for the maximal level reached before time $\tau_i$, when $i\ge 1$, with the convention $M_0=1$. 
We will use Proposition \ref{prop:main}, but one has to take care that the result holds only for $x$ large enough. 
Indeed, the proof reveals that there are two conditions that $x$ should satisfy. First the event $\mathcal E_x$ should hold, with the notation from Lemma \ref{exn}, and secondly 
$x$ should be larger than $2^{11} |\Gamma| x_0$, where $x_0$ is random, but by Proposition \ref{prop:wall}, we know that it is stochastically dominated by $\delta^{1/4}$ times a geometric random variable with parameter $1/2$. Write $\widetilde {\mathcal E}_x$ for the event when these two conditions are satisfied, and note two things: first it is $\mathcal F_{H_x}$-measurable, and 
by Lemma \ref{exn} one has $\bP(\widetilde {\mathcal E}_x)\ge 1- \exp(-cx)$, for some constant $c>0$ (depending only on $|\Gamma|$ and $\delta$). 
Note that defining $\widetilde {\mathcal E}_x^+ := \cap_{y\ge x} \widetilde {\mathcal E}_x$, we also get by a union bound that $\bP(\widetilde {\mathcal E}_x^+)\ge 1- C\exp(-cx)$, for some constant $C>0$.

Now fix some $i\ge 1$. 
The proof of Proposition \ref{prop:main} and the previous discussion, show that for all $n\ge 1$, and for some constant $C>0$,     
\begin{align*}
 \bP\left(M_{i+1}\ge 2^n M_i\right)\ & \le\ \bP\left(M_{i+1}\ge 2^n M_i, \, \widetilde {\mathcal E}_{2^{n-1}M_i}\right) +\bP(\widetilde {\mathcal  E}_{2^{n-1}M_i}^c) \\
  &\le \ \frac{1}{2^{10}}\, \bP\left(M_{i+1}\ge 2^{n-1} M_i\right) +\bP((\widetilde {\mathcal E}_{2^{n-1}}^+)^c) \\
  & \le \ \frac{1}{2^{10}}\, \bP\left(M_{i+1}\ge 2^{n-1} M_i\right) + C\exp(-c2^{n-1}),
\end{align*} 
using that $M_i\ge 1$ at the second line. 
By induction we obtain that for all $n\ge 1$, 
\begin{equation}\label{Mi}
 \bP\left(M_{i+1}\ge 2^n M_i \right) \  \le\  C\cdot 2^{-10 n},
\end{equation}
for some possibly different constant $C>0$ (still depending on $\delta$). 
Then by Cauchy-Schwarz,  
\begin{align}\label{taui}
\nonumber \E[\tau_{i+1} - \tau_i^+ ] \ & \le \ \sum_{n\ge 1}  \E[(\tau_{i+1} - \tau_i^+) {\bf 1}(  2^{n-1}M_i\le M_{i+1}\le 2^nM_i) ]  \\
&\le  \ C\cdot \sum_{n\ge 1} \frac{1}{2^{5n}}\,  \E[(\tau_{i+1} - \tau_i^+)^2 {\bf 1}( M_{i+1}\le 2^nM_i)  ]^{1/2}.
\end{align}
Now by definition on the event 
$\{M_{i+1}\le 2^nM_i\}$, one has $\tau_{i+1}\le \widetilde \tau_{i+1}$, where $\widetilde \tau_{i+1}$ is the first time after $\tau_i^+$ when the walk reaches either level $0$ or level $2^nM_i$. 
Therefore, almost surely 
\begin{equation}\label{tautildei}
\E[(\tau_{i+1} - \tau_i^+)^2 {\bf 1}( M_{i+1}\le 2^nM_i)  ]\ \le\ \E[(\widetilde \tau_{i+1} - \tau_i^+)^2  ].
\end{equation} 
Next, for any time $t$ satisfying $\tau_i^+ < t < \widetilde \tau_{i+1}$, one of the two following cases may hold. 
Either all edges between levels $0$ and $2^nM_i$ have already been discovered, 
in which case it is standard that the hitting time of level $0$ or $2^nM_i$ is bounded by a constant (depending on $\Gamma$) times $2^{4n}M_i^2$. 
If not, at least one edge has not been crossed yet. In this case, a basic coupling with the simple random walk shows that the second moment of $H_{\partial_v\mathcal R_t}\wedge H_0\wedge H_{2^nM_i}$ is also bounded by a constant times $2^{4n}M_i^2$. But the process will cross a new edge after a geometric number of visits to $\partial_v\mathcal R_t$. Thus  
in any case,  
the second moment of the time needed to either cross one new edge, or hit level $0$ or level $2^nM_i$ is bounded by a constant times $2^{4n}M_i^2$. 
Since the total number of edges between these two levels is bounded by $|\Gamma|^2 2^nM_i$, using \eqref{tautildei}, we get the bound 
$$\E[(\tau_{i+1} - \tau_i^+)^2 {\bf 1}( M_{i+1}\le 2^nM_i)  ] \le C \cdot 2^{6n}\, \mathbb E[M_i^3],$$
for some constant $C$ (that depends on $\delta$ and $\Gamma$). 
Injecting this in \eqref{taui}, we get 
$$ \E[\tau_{i+1} - \tau_i^+  ] \ \le\ C\cdot \E[M_i^3]^{1/2}. $$
Finally applying \eqref{Mi} again, we deduce that the third moment of $M_i$ is finite.  
Using then symmetry with respect to level $0$, concludes the proof. 
\end{proof}

\end{document}